\documentclass[11pt]{article}

\usepackage{amsmath, amssymb, amsthm, graphicx, color}
\usepackage{marginnote}
\usepackage{hyperref}
\usepackage{booktabs}

\newtheorem{thm}{Theorem}

\newcommand{\armh}{L_{\textrm{horiz}}}
\newcommand{\armv}{L_{\textrm{vert}}}
\newcommand{\hallway}{L}
\newcommand{\rotmat}[1]{R_{#1}}
\newcommand{\R}{\mathbb{R}}
\newcommand{\nvecone}{\boldsymbol{\mu}}
\newcommand{\nvectwo}{\boldsymbol{\nu}}
\newcommand{\contact}{\Gamma}

\title{Differential equations and exact solutions in the moving sofa problem}
\author{Dan Romik\footnote{Department of Mathematics, University of California, Davis, One Shields Ave, Davis, CA 95616, USA. Email: \texttt{romik@math.ucdavis.edu}}}
\date{July 10, 2016}

\begin{document}
\maketitle

\begin{abstract}
The moving sofa problem, posed by L.~Moser in 1966, asks for the planar shape of maximal area that can move around a right-angled corner in a hallway of unit width, and is conjectured to have as its solution a complicated shape derived by Gerver in 1992. We extend Gerver's techniques by deriving a family of six differential equations arising from the area-maximization property. We then use this result to derive a new shape that we propose as a possible solution to the \emph{ambidextrous moving sofa problem}, a variant of the problem previously studied by Conway and others in which the shape is required to be able to negotiate a right-angle turn both to the left and to the right. Unlike Gerver's construction, our new shape can be expressed in closed form, and its boundary is a piecewise algebraic curve. 
Its area is equal to $X+\arctan Y$, where $X$ and $Y$ are solutions to the cubic equations $x^2(x+3)=8$ and $x(4x^2+3)=1$, respectively.
\end{abstract}

\renewcommand{\thefootnote}{\fnsymbol{footnote}} 
\renewcommand{\thefootnote}{\arabic{footnote}}

\section{Introduction}

\label{sec:introduction}

\begin{quote}
\textit{``Odd,'' agreed Reg. ``I've certainly never come across any irreversible mathematics involving sofas. Could be a new field. Have you spoken to any spatial geometricians?''}

\hfill ---Douglas Adams, ``Dirk Gently's Holistic Detective Agency'' \\[13pt]
\end{quote}

\subsection{Background: the moving sofa problem}

The humorist and science-fiction writer Douglas Adams, whose 1987 novel \cite{adams} featuring a sofa stuck in a staircase is quoted above, was not the first to observe that the geometric intricacies of moving sofas around corners and other obstacles, familiar from our everyday experience, raise some challenging mathematical questions. In 1966, the mathematician Leo Moser asked~\cite{moser} the following curious question, which came to be known as the \textbf{moving sofa problem}:

\begin{quote}
\textit{What is the planar shape of maximal area that can be moved around a right-angled corner in a hallway of unit width?}
\end{quote}

Fifty years later, the problem is still unsolved. Thanks to its whimsical nature and the surprising contrast between the ease of stating and explaining the problem and the apparent difficulty of solving it, it has been mentioned in several books \cite{unsolved-problems, finch, stewart, weisstein}, 
has dedicated pages describing it on Wikipedia \cite{wikipedia} and Wolfram MathWorld \cite{mathworld}, 
and, especially in recent years, has been a popular topic for discussion online in math-themed blogs \cite{baez1, baez2, goucher, pachter, ross} and online communities \cite{reddit, mathoverflow}. 
(As further illustrations of its popular appeal, the moving sofa problem is the first of three open problems mentioned on the back cover of Croft, Falconer and Guy's book \cite{unsolved-problems} on $148$ unsolved problems in geometry, and is currently the third-highest-voted open problem from among a list of 99 ``not especially famous, long-open problems which anyone can understand'' compiled by participants of the online math research community \nobreak{MathOverflow}~\cite{mathoverflow}.)
However, as those who have studied it can appreciate, and as we hope this paper will convince you, the problem is far more than just a curiosity, and hides behind its simple statement a remarkable amount of mathematical structure and depth.

Fig.~\ref{fig:semicircle-hammersley}(a) shows two trivial examples of shapes that can be moved around a corner --- a unit square (with area $1$) and a semicircle of unit radius (with area $\pi/2 \approx 1.57$). The latter example is more interesting, since moving the semicircle around the corner requires translating it, then rotating it by an angle of $\pi/2$ radians, then translating it again, whereas moving the square involves only translations. By \emph{simultaneously} performing translations and rotations it is not hard to improve on these trivial constructions, and indeed the key to maximizing the area is to find the precise sequence that combines those two rigid motions in the optimal way.

\begin{figure}[h]
\begin{center}
\begin{tabular}{cc}
\scalebox{0.4}{\includegraphics{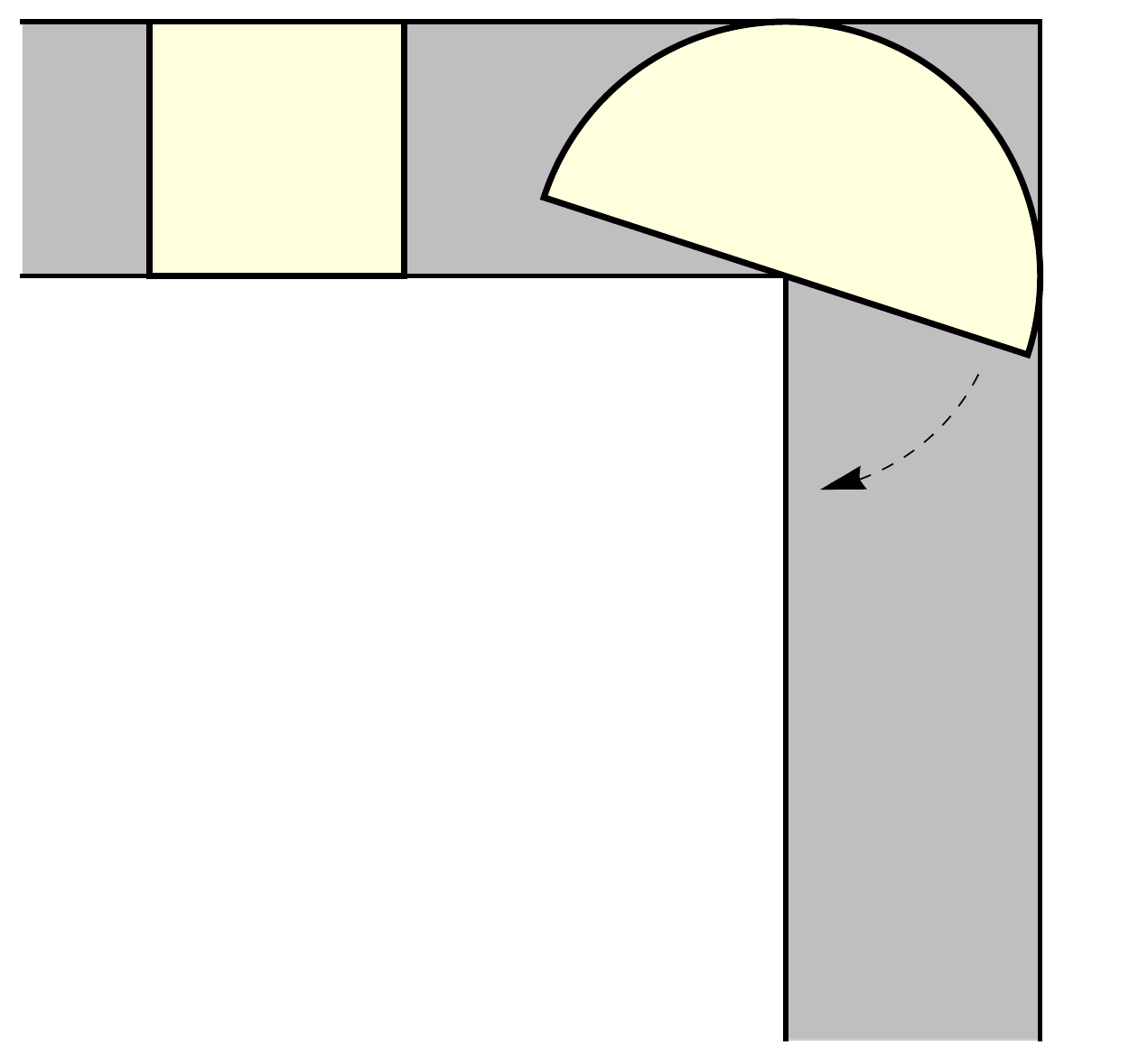}}
&
\raisebox{30pt}{\scalebox{0.6}{\includegraphics{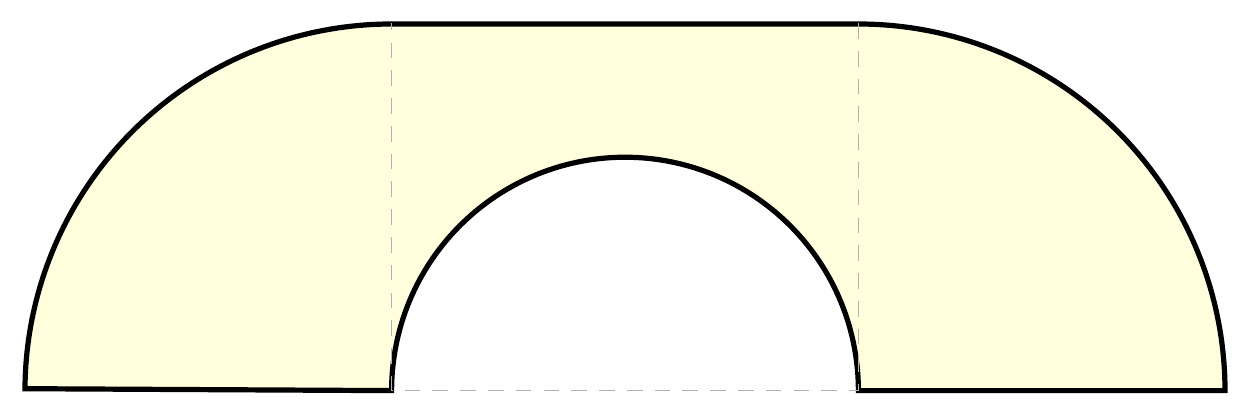}}}
\\(a)&(b)
\end{tabular}
\caption{(a) The $L$-shaped hallway and two trivial sofa shapes that can move around a corner; (b)~Hammersley's sofa.}
\label{fig:semicircle-hammersley}
\end{center}
\end{figure}

It is known that a shape of maximal area in the moving sofa problem exists (a result attributed to Conway and M.~Guy \cite{unsolved-problems}, though Gerver's proof in \cite{gerver} is the only one we are aware of that appeared in print). Hammersley in 1968 showed that the maximal area is at most $2\sqrt{2}\approx 2.828$ \cite{hammersley} (see also \cite{stewart, wagner}), and proposed a shape of area $\frac{\pi}{2}+\frac{2}{\pi} \approx 2.2074$ comprising two unit radius quarter-circles separated by a rectangular block of dimensions $\frac{4}{\pi}\times 1$ with a semicircular piece of radius $\frac{2}{\pi}$ removed (Fig.~\ref{fig:semicircle-hammersley}(b)). He conjectured this shape to be optimal, but this was discovered to be false when constructions of slightly larger area were discovered \cite{guy-monthly}. In 1992, Gerver \cite{gerver} proposed a considerably more complicated shape that can move around the corner, whose boundary comprises 3 straight line segments and 15 distinct curved segments, each described by a separate analytic expression; see Fig.~\ref{fig:intro-gerver}. (As recounted by Stewart in \cite{stewart}, the same solution had been found earlier in 1976 by B.~F.~Logan but never published.) Gerver conjectured that his proposed shape, derived from considerations of local optimality and having area $2.21953166\ldots$, has maximal area, and to date no constructions with larger area have been found.

\begin{figure}
\begin{center}
\scalebox{0.6}{\includegraphics{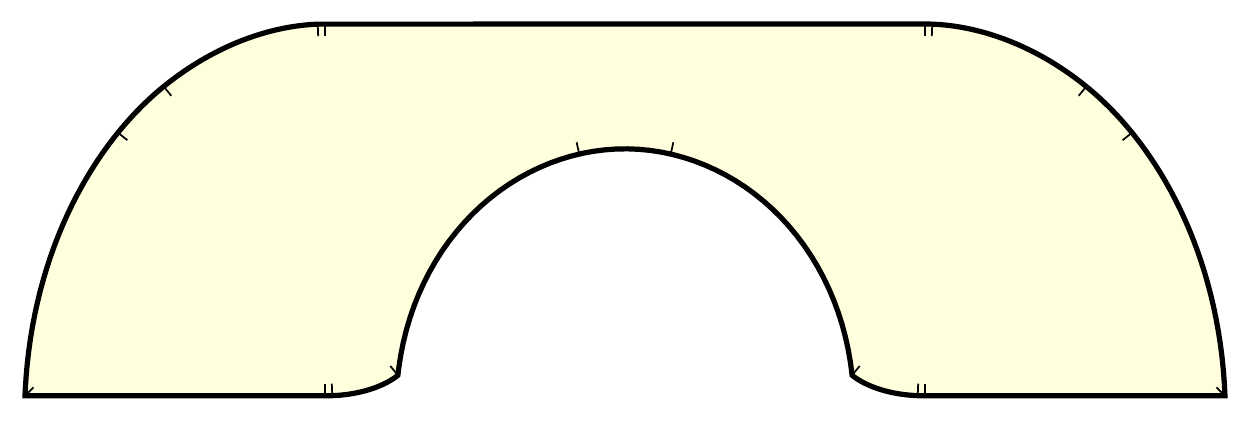}}
\caption{Gerver's sofa. The tick marks delineate the transition points between distinct pieces of the boundary --- $3$ straight line segments and $15$ curved pieces, each of which is described by a separate analytic expression.}
\label{fig:intro-gerver}
\end{center}
\end{figure}

It is worth noting that Gerver's description of his shape is not fully explicit, in the sense that the analytic formulas for the curved pieces of the shape are given in terms of four numerical constants $A$, $B$, $\varphi$ and $\theta$ (where $0<\varphi<\theta<\pi/4$ are angles with a certain geometric meaning), which are defined only implicitly as solutions of the nonlinear system of four equations
\begin{align}
&A(\cos \theta-\cos\varphi)-2B\sin \varphi
\nonumber \\ & \qquad\qquad\qquad +(\theta-\varphi-1)\cos\theta-\sin\theta+\cos\varphi+\sin\varphi &=0, \qquad\  \label{eq:gerver-nonlinear1} \\
&A(3\sin\theta+\sin\varphi)-2B\cos\varphi
\nonumber \\ & \qquad\qquad\qquad+3(\theta-\varphi-1)\sin\theta+3\cos\theta-\sin\varphi+\cos\varphi
\!\!\!&=0, \qquad \ 
\label{eq:gerver-nonlinear2} \\
&A\cos\varphi-(\sin\varphi+\tfrac12-\tfrac12\cos\varphi+B\sin\varphi)&=0, \qquad \ 
\label{eq:gerver-nonlinear3} \\
&(A+\tfrac12\pi-\varphi-\theta)-\big((B-\tfrac12(\theta-\varphi)(1+A)-\tfrac14(\theta-\varphi)^2\big)&=0.\qquad \ 
\label{eq:gerver-nonlinear4} 
\end{align}
Actually these equations are linear in $A$, $B$, so these auxiliary variables can be eliminated, leaving just two transcendental equations for the angles $\varphi$ and $\theta$, but this is as much of a simplification as one can get, and even then, Gerver's shape, its area, and other interesting quantities associated with it (such as its length from left to right, the arc length of its boundary, and the angles $\varphi$ and $\theta$ themselves) cannot be expressed in closed form. 

\subsection{Main result: the ambidextrous moving sofa problem}

Around the time that the moving sofa problem was first published, John H. Conway, G.~C.~Shepard and several other mathematicians were said to have worked on the problem during a geometry conference, as well as on several other variants of the problem, each of which was apparently assigned to one member of the group  \cite{stewart}. We now consider one of these variants, which asks for the planar shape of maximal area that can negotiate right-angled turns \emph{both to the right and to the left} in a hallway of width $1$. We refer to this as the \textbf{ambidextrous moving sofa problem}. The Conway et al.~early attack led to two rough guesses about the approximate shape of the solution, nicknamed the ``Conway car'' and ``Shepard piano'' (Fig.~\ref{fig:maruyama-gibbs}(a)--(c)).
The problem was considered again in 1973 by Maruyama \cite{maruyama}, who developed a numerical scheme for computing polygonal approximations to the problem and several other variants (Fig.~\ref{fig:maruyama-gibbs}(d)). 
More recently, in a 2014 paper Gibbs \cite{gibbs} developed another numerical technique to study the problem, and computed a similar-looking shape (in much higher resolution), whose area he calculated to be approximately $1.64495$. Gibbs' shape is shown in Fig.~\ref{fig:maruyama-gibbs}(e).

\begin{figure}
\begin{center}
\begin{tabular}{ccc}
\scalebox{0.1}{\includegraphics{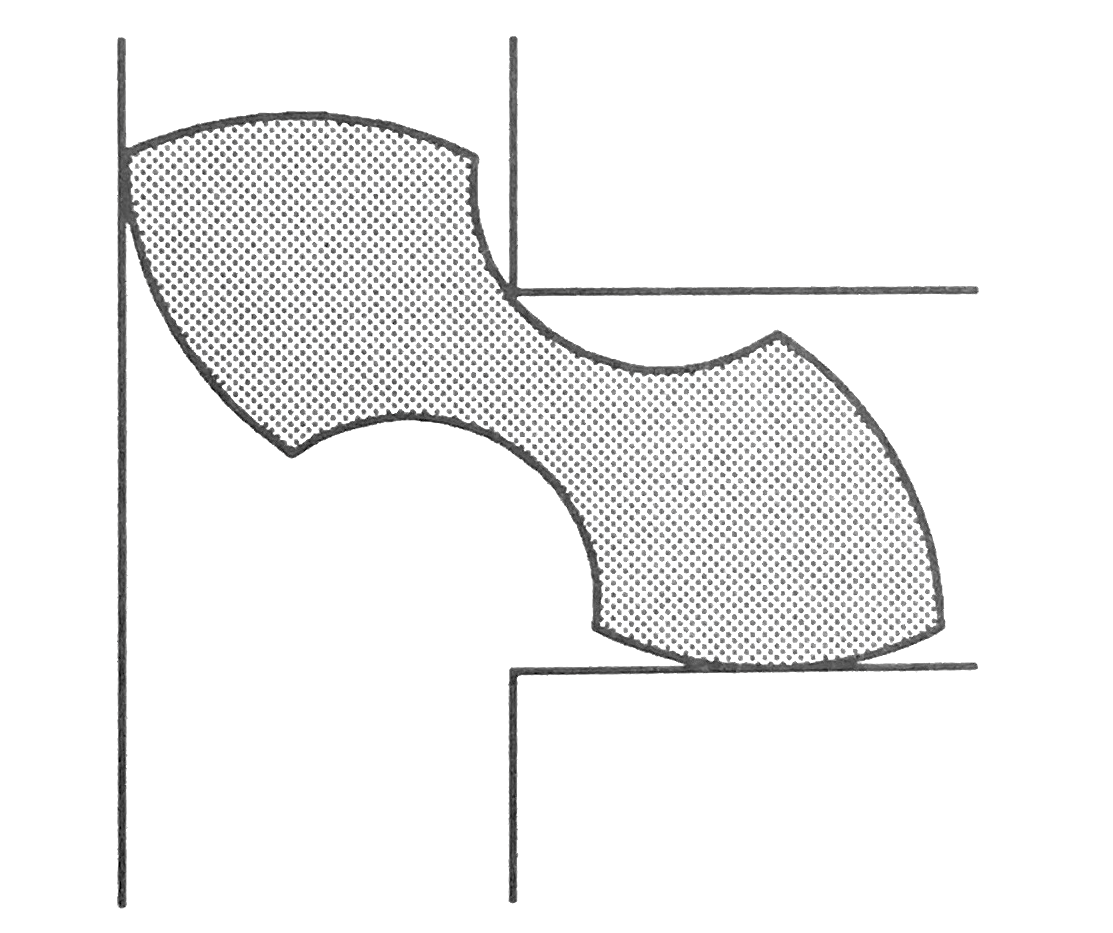}} 
& 
\scalebox{0.1}{\includegraphics{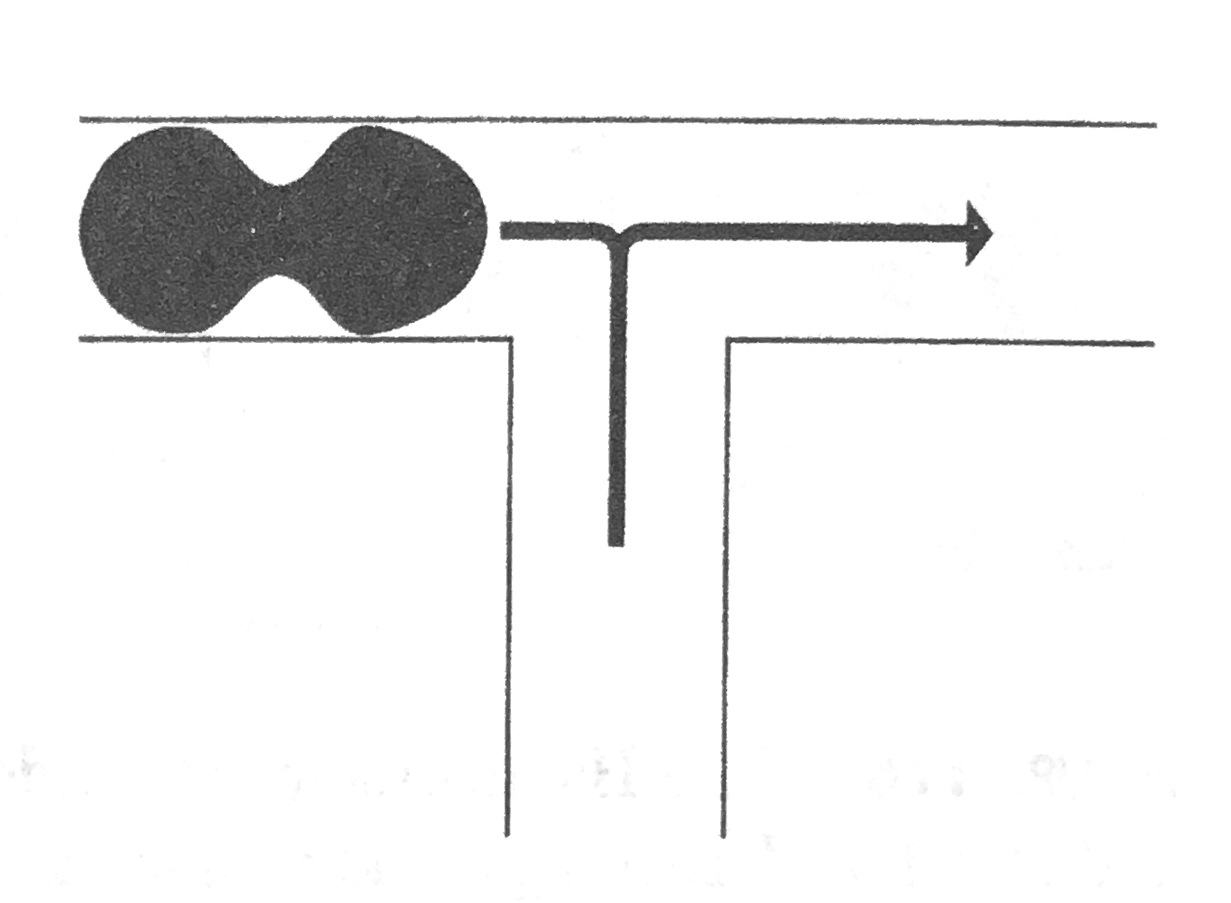}} 
&
\scalebox{0.1}{\includegraphics{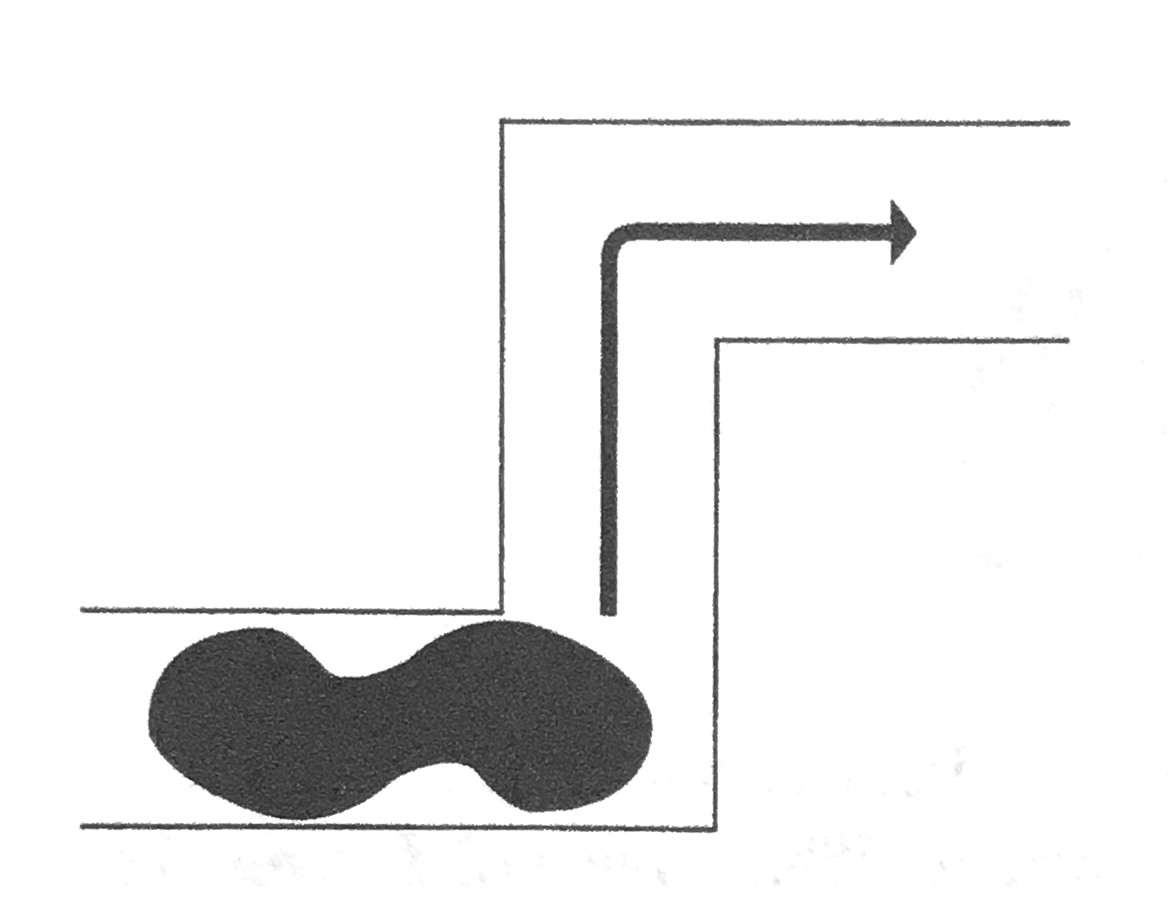}} 
\\ (a) & (b) & (c)
\end{tabular}
\begin{tabular}{cc}
\scalebox{0.32}{\includegraphics{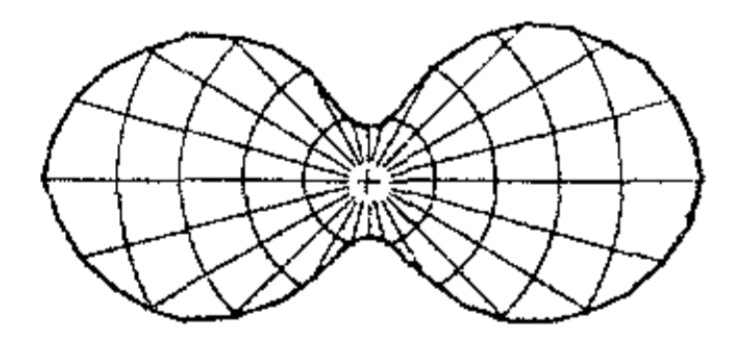}} 
& 
\scalebox{0.105}{\includegraphics{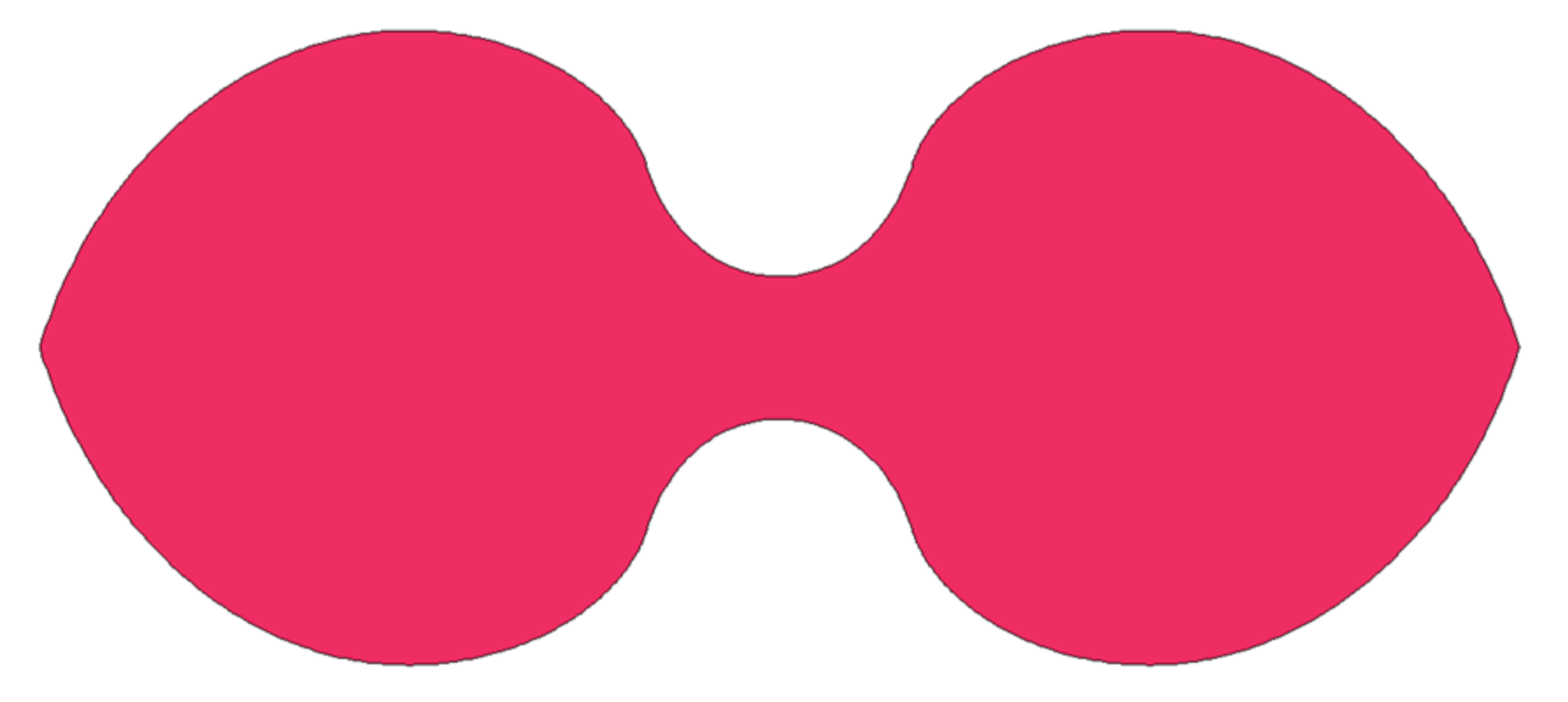}} 
\\
\!\!(d) & (e)
\end{tabular}
\caption{
Five graphic figures reproduced from \cite{unsolved-problems, gibbs, maruyama, stewart} showing past attempts to attack the ambidextrous moving sofa problem and another version of it asking for the largest shape that can turn completely around in a T-junction (the answer to both versions may be the same): (a) a shape dubbed the ``Conway car'' in \cite{unsolved-problems}; (b) another shape from \cite{stewart}, also referred to as the Conway car; (c) another shape from \cite{stewart}, where it is referred to as the ``Shepard piano''; (d) an approximate shape computed numerically by Maruyama \cite{maruyama}; (e) another approximate polygonal shape computed numerically by Gibbs~\cite{gibbs}.
}
\label{fig:maruyama-gibbs}
\end{center}
\end{figure}

Our main result is the construction of a precisely-defined shape that satisfies the conditions of an ambidextrous moving sofa (i.e., it can move around corners to the left and to the right) and is derived from considerations of local optimality analogous to Gerver's shape, and hence is a plausible candidate to be the solution to the ambidextrous moving sofa problem. Our shape, shown in Fig.~\ref{fig:intro-ambidextrous},
appears visually indistinguishable from Gibbs' numerically computed approximate shape. Its boundary comprises 18 distinct segments, each given by a separate explicit formula. Moreover, to our surprise we found that, unlike the case of Gerver's sofa, the new shape and all of its associated parameters can be described in closed form, and its boundary segments are all pieces of algebraic curves (see Fig.~\ref{fig:algebraic-curves} in Section~\ref{sec:geom-alg-properties}). In particular, its area is given by the rather unusual explicit constant
\begin{align}
&\sqrt[3]{3+2 \sqrt{2}}+\sqrt[3]{3-2 \sqrt{2}}-1 
+\arctan\left[
\frac{1}{2} \left( \sqrt[3]{\sqrt{2}+1}- \sqrt[3]{\sqrt{2}-1}\, \right)
  \right]
\nonumber \\ & \qquad\qquad\qquad \approx 1.644955218425440,
\label{eq:ambi-area-closedform}
\end{align}
a result which is nicely in accordance with Gibbs' earlier numerical prediction. The left-to-right length of our new shape is 
\begin{equation}
\label{eq:ambi-length-closedform}
\frac{2}{3} \sqrt{4+\sqrt[3]{71+8 \sqrt{2}}+\sqrt[3]{71-8 \sqrt{2}}}
\approx 2.334099633100619,
\end{equation}
an algebraic number of degree $6$.
Section~\ref{sec:ambidextrous}, where the details of our derivation are given, contains additional curious formulas of this sort, and Section~\ref{sec:geom-alg-properties} has a further discussion of geometric and algebraic properties of the shape.

\begin{figure}
\begin{center}
\scalebox{0.8}{\includegraphics{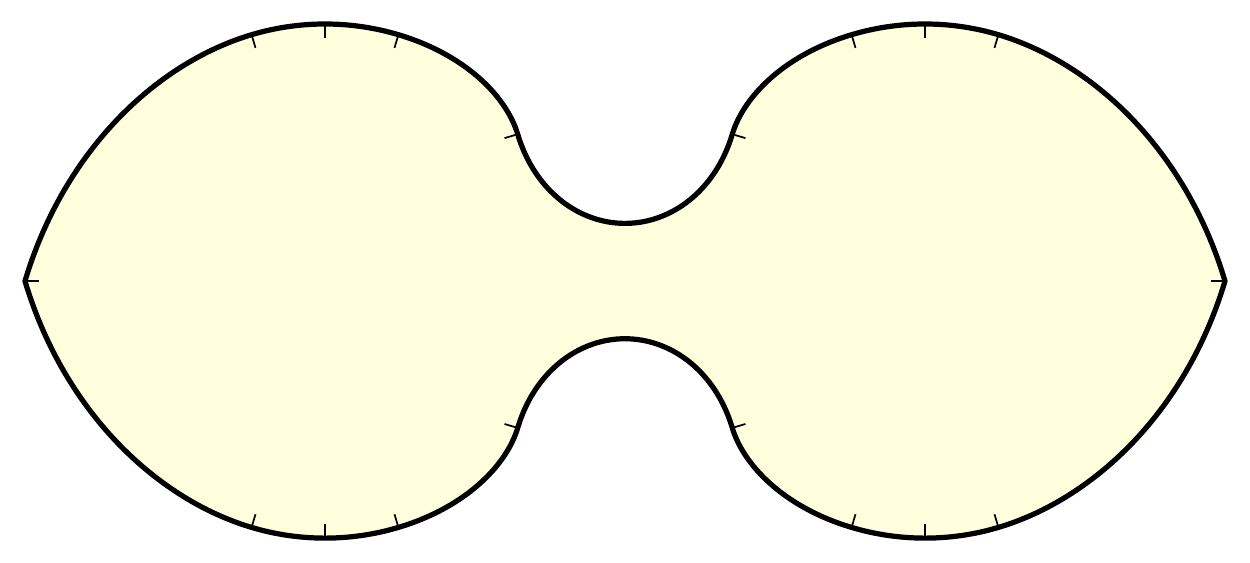}}
\caption{Our new analytic construction in the ambidextrous moving sofa problem. The tick marks show the subdivision of the boundary of the shape into 18 distinct curves, each described by a separate analytic expression.}
\label{fig:intro-ambidextrous}
\end{center}
\end{figure}

\subsection{Additional results}

On the way to deriving the new shape, this paper makes several additional contributions to the understanding of the moving sofa problem and its variants. A main new advance in the theory is the development of a general framework extending and generalizing Gerver's ideas. Two key elements of this new framework are: 

\begin{itemize}
\item First, we define a convenient terminology and notation that parametrizes candidate sofa shapes in terms of the so-called \textbf{rotation path}, which is the path traversed by the inner corner of the hallway as it slides and rotates around the shape in a sofa-centric frame of reference in which the shape stays fixed and it is the hallway that moves and rotates. We give an explicit description of this parametrization by deriving formulas for the boundary of the shape associated with a given rotation path. These ideas are described in Section~\ref{sec:rotation-paths}.

\item Second, by starting with Gerver's key insight that a shape moving around a corner can have maximal area only if it is a limit of polygonal shapes satisfying a certain geometric condition (what Gerver refers to as a \textbf{balanced polygon}) and reconsidering it from our new point of view, we rework Gerver's balance condition into a family of six ordinary differential equations. We show that it is a necessary condition for the rotation path associated with an area-maximizing moving sofa shape to satisfy these ODEs, subject to certain mild assumptions.
The ODEs are derived and discussed in Section~\ref{sec:odes}.
\end{itemize}

After developing this new framework, we illustrate its applicability by first giving a new and conceptually quite simple derivation of Gerver's shape, which consists of writing down an intuitive and easy-to-understand system of equations for gluing together solutions to five of the six ODEs. The equations are then easily solved by a computer using the symbolic math software application \texttt{Mathematica}. This is discussed in Section~\ref{sec:gerver-rederivation}. In Section~\ref{sec:ambidextrous} we then show how the same ideas can then be applied with slight modifications to derive the new shape in the ambidextrous moving sofa problem.

Several of the proofs in this paper rely on numerical and algebraic computations that can be performed by a computer using symbolic math software. We prepared a \texttt{Mathematica} software package, \texttt{MovingSofas}, as a companion package to this paper to aid the reader in the verification of a few of the claims~\cite{romik-sofa-mathematica}. The software package also includes interactive graphical visualizations and video animations that can greatly enhance the intuitive understanding of the geometry of shapes moving around a corner; see also~\cite{romik-sofa-web}. 

\subsection*{Acknowledgements}

The author wishes to thank Greg Kuperberg, Alexander Coward, Alexander Holroyd, James Martin and Anastasia Tsvietkova for helpful conversations and suggestions during the work described in this paper. The author was supported by the National Science Foundation under grant DMS-0955584.

\section{Rotation paths, contact points and contact paths}

\label{sec:rotation-paths}

We equip the $L$-shaped hallway and its two arms with coordinates by denoting
\begin{align*}
\armh &= \{(x,y)\in\R^2\,:\, x\le 1,\ 0\le y\le 1 \}, \\
\armv &= \{(x,y)\in\R^2\,:\, y\le 1,\ 0\le x\le 1 \}, \\
\hallway &= \armh\cup\armv.
\end{align*}
The moving sofa problem considers shapes that undergo a rigid motion (a combination of rotations and translations) to move continuously from $\armh$ into $\armv$, while staying within $\hallway$. In this paper we further restrict our attention to shapes that, while being translated, rotate monotonically from an angle of $0$ radians to an angle of $\pi/2$ radians. 
It has not been shown rigorously, but seems extremely plausible, that the optimal shape can be transported around the corner using a rigid motion of this type.

It will be convenient to also keep in mind (as several earlier authors have done) a dual point of view from the frame of reference of the sofa, in which it is the hallway being rotated and translated whereas the sofa stays in a fixed place. From this point of view we see that a shape $S$ can be moved around the corner if it satisfies the condition
\begin{equation} \label{eq:hallway-containment}
S \subseteq \armh \cap \bigcap_{0\le t\le \pi/2} \Big(\mathbf{x}(t)+\rotmat{t}(\hallway)\Big)
\cap \Big(\mathbf{x}(\pi/2)+\rotmat{\pi/2}(\armv) \Big),
\end{equation}
where we denote by $\rotmat{t}$ the rotation matrix 
$$\rotmat{t}=\begin{pmatrix} \cos t&-\sin t\\\sin t&\cos t\end{pmatrix},$$ and where $\mathbf{x}:[0,\pi/2]\to\R^2$ is a continuous path satisfying $\mathbf{x}(0)=(0,0)^\top$.\footnote{Throughout the paper, we denote vectors in $\R^2$ with boldface letters and consider them as column vectors.} The path~$\mathbf{x}$ describes the motion of the inner corner $(0,0)$ of the hallway in the frame of reference of the sofa, with the time parameter being the angle of rotation of the hallway. We refer to any such path as a \textbf{rotation path}.

\begin{figure}
\begin{center}
\scalebox{0.4}{\includegraphics{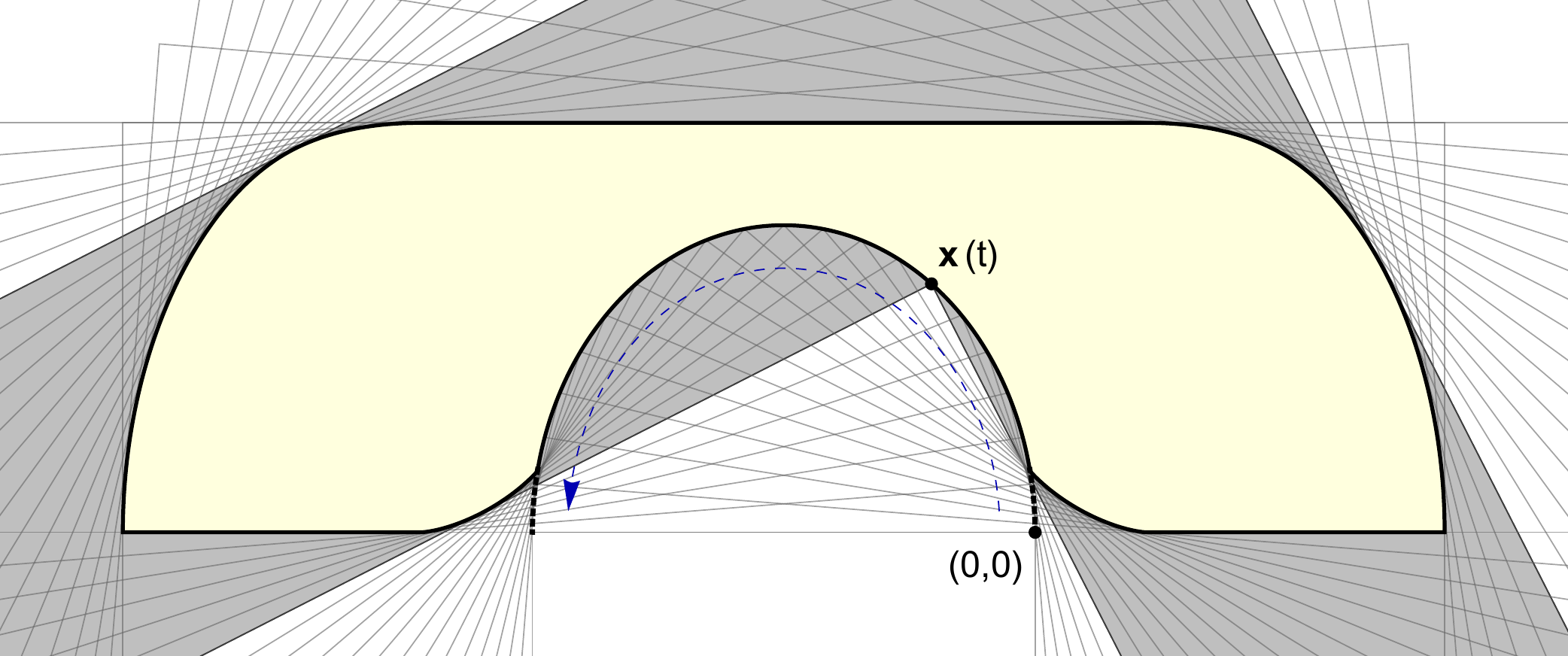}}
\caption{Constructing a shape from a rotation path.}
\label{fig:rotation-path}
\end{center}
\end{figure}

Now note that, given a rotation path $\mathbf{x}$, there is no loss of generality in considering only shapes $S$ such that the containment relation in \eqref{eq:hallway-containment} is actually an equality, since in the case of a strict containment one can replace the shape by a bigger one with a larger (or equal) area. We therefore define
\begin{equation} \label{eq:rotation-shape}
S_{\mathbf{x}} = \armh \cap \bigcap_{0\le t\le \pi/2} \Big(\mathbf{x}(t)+\rotmat{t}(\hallway)\Big)
\cap \Big(\mathbf{x}(\pi/2)+\rotmat{\pi/2}(\armv) \Big),
\end{equation}
and refer to this set as the shape associated with the rotation path $\mathbf{x}$; see Fig.~\ref{fig:rotation-path}. Parametrizing shapes in such a way in terms of their rotation paths, the problem is now reduced to identifying the rotation path whose associated shape has maximal area.

A natural question now arises of describing the map $\mathbf{x}\mapsto S_{\mathbf{x}}$ associating shapes to rotation paths. Giving a fully explicit description of this map seems challenging (and therein perhaps lies a key difficulty to solving the moving sofa problem) due to the complicated and rather subtle effect that a local change to the rotation path can have on several different parts of the shape $S_{\mathbf{x}}$. However, we can give a partial description that is valid for relatively simple rotation paths and is already quite useful. The key idea is to keep track of the \textbf{contact points}, which are the tangency points between the four walls of the hallway and the shape, in addition to the position $\mathbf{x}$ of the rotating hallway corner, which is also considered a contact point if it touches the shape $S_{\mathbf{x}}$. We label these additional tangency points $\mathbf{A}$, $\mathbf{B}$, $\mathbf{C}$ and $\mathbf{D}$, where $\mathbf{A}$ and $\mathbf{C}$ correspond to the outer walls and $\mathbf{B}$ and $\mathbf{D}$ correspond to the inner walls, as shown in Fig.~\ref{fig:contact-paths}. As the hallway slides and rotates, the tangency points of the four walls trace four paths, which we refer to as the \textbf{contact paths}, and denote by $\mathbf{A}(t)$, $\mathbf{B}(t)$, $\mathbf{C}(t)$ and $\mathbf{D}(t)$. Note that for any given value of $t$, one or both of the contact points $\mathbf{B}$ and $\mathbf{D}$ may not exist, and $\mathbf{x}$ may or may not be a contact point. We denote by $\contact_\mathbf{x}(t)$ the set of contact points; for example, in the situation shown in Fig.~\ref{fig:contact-paths} we have $\contact_\mathbf{x}(t)=\{\mathbf{x}, \mathbf{A}, \mathbf{B}, \mathbf{C}, \mathbf{D} \}$. 

Note also that for certain choices of rotation paths and values of $t$, one or both of the inner hallway walls may have more than one tangency point with the shape, in which case our notation $\mathbf{B}$ and $\mathbf{D}$ breaks down. We do not consider such more complicated situations. 

\begin{figure}
\begin{center}
\scalebox{0.4}{\includegraphics{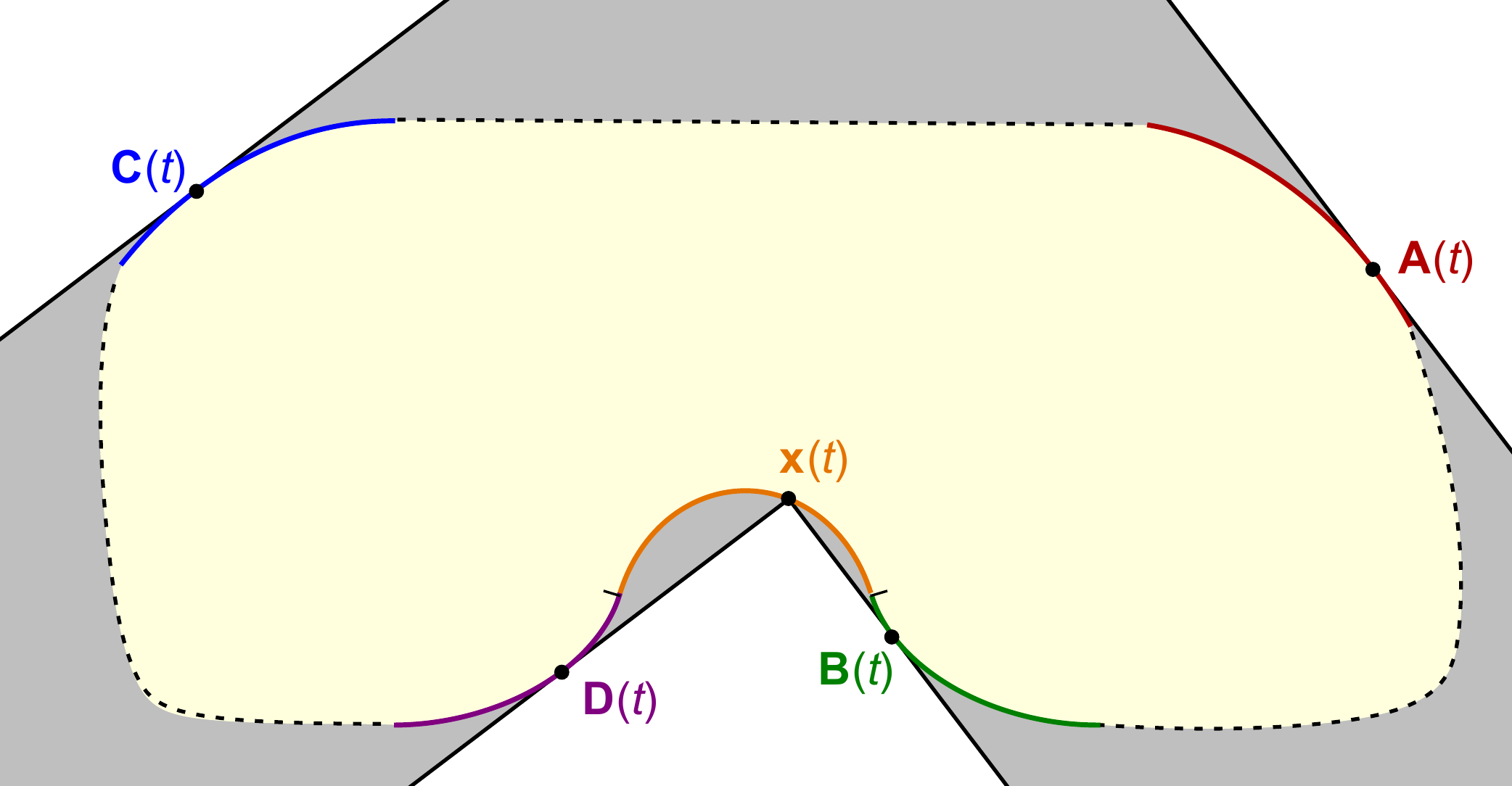}}
\caption{The contact points and contact paths $\mathbf{x}(t), \mathbf{A}(t)$, $\mathbf{B}(t)$, $\mathbf{C}(t)$, $\mathbf{D}(t)$.}
\label{fig:contact-paths}
\end{center}
\end{figure}

Denote
$$
\nvecone_t = \begin{pmatrix}\cos t\\\sin t\end{pmatrix} = \rotmat{t} \begin{pmatrix}1\\0\end{pmatrix},
\qquad
\nvectwo_t = 
\begin{pmatrix}-\sin t\\\cos t\end{pmatrix} = \rotmat{t} \begin{pmatrix}0\\1\end{pmatrix},
$$
a rotating orthonormal frame.
Our first result gives an explicit description of the contact paths $\mathbf{A}, \mathbf{B}, \mathbf{C}, \mathbf{D}$ in terms of the rotation path $\mathbf{x}$, under the assumptions described above.

\begin{thm} 
\label{thm:contact-paths}
We have the relations
\begin{align}
\mathbf{A}(t) &= \mathbf{x}(t) + \Big\langle \mathbf{x}'(t),\nvecone_t \Big\rangle \,\nvectwo_t + \nvecone_t, 
\label{eq:contact-pathA} \\[3pt]
\mathbf{B}(t) &= \mathbf{x}(t) + \Big\langle \mathbf{x}'(t), \nvecone_t \Big\rangle \, \nvectwo_t,
\label{eq:contact-pathB} \\[3pt]
\mathbf{C}(t) &= \mathbf{x}(t) - \Big\langle \mathbf{x}'(t), \nvectwo_t \Big\rangle \, \nvecone_t + \nvectwo_t 
\label{eq:contact-pathC} \\[3pt]
\mathbf{D}(t) &= \mathbf{x}(t) - \Big\langle \mathbf{x}'(t), \nvectwo_t \Big\rangle \, \nvecone_t,
\label{eq:contact-pathD}
\end{align}
which are valid whenever the respective contact points are defined, the respective contact path is continuous at $t$ and $\mathbf{x}$ is differentiable at $t$.
\end{thm}

\begin{proof} 
By definition, the contact point $\mathbf{A}(t)$ lies on the line
$$
\ell_t := \left\{ \mathbf{p}\,:\,
\Big\langle \mathbf{p}, \nvecone_t \Big\rangle = \Big\langle \mathbf{x}(t), \nvecone_t \Big\rangle +1
\right\}.
$$
Denote $s=t+\delta$, where $\delta$ is a small positive number.
Let $\mathbf{p}(t,s)$ denote the point at the intersection of the lines $\ell_t$ and $\ell_s$. We will derive \eqref{eq:contact-pathA} starting from the fact that
$$ \mathbf{A}(t) = \lim_{\delta\to 0} \mathbf{p}(t,s), $$
which is immediate from the assumption that $\mathbf{A}(\cdot)$ is continuous at $t$. To this end, note that $\mathbf{p}(t,s)$ satisfies
\begin{align}
\Big\langle \mathbf{p}(t,s), \nvecone_t \Big \rangle &= \Big\langle \mathbf{x}(t), \nvecone_t \Big \rangle+1, 
\label{eq:ptdelta1}
\\
\Big\langle \mathbf{p}(t,s), \nvecone_s \Big \rangle &= \Big\langle \mathbf{x}(s), \nvecone_s \Big \rangle+1.
\label{eq:ptdelta2}
\end{align}
Furthermore, $\nvecone_s$ can be represented as
$$
\nvecone_s = \cos \delta \cdot \nvecone_t + \sin \delta \cdot \nvectwo_t,
$$
which allows us to rewrite \eqref{eq:ptdelta2} as
$$
\cos \delta \Big\langle \mathbf{p}(t,s), \nvecone_t \Big \rangle 
+ \sin \delta \Big\langle \mathbf{p}(t,s), \nvectwo_t \Big \rangle
= 
\cos \delta \Big\langle \mathbf{x}(s), \nvecone_t \Big \rangle
+\sin \delta \Big\langle \mathbf{x}(s), \nvectwo_t \Big \rangle
+1.
$$
Using \eqref{eq:ptdelta1}, this equation transforms into
$$
\sin \delta \Big\langle \mathbf{p}(t,s), \nvectwo_t \Big \rangle 
= 
\cos \delta \Big\langle \mathbf{x}(s)-\mathbf{x}(t), \nvecone_t \Big \rangle
+\sin \delta \Big\langle \mathbf{x}(s), \nvectwo_t \Big \rangle.
$$
Dividing by $\delta$ and taking the limit as $\delta\to 0$, we get that
$$
\Big\langle \mathbf{A}(t), \nvectwo_t \Big\rangle = \Big\langle \mathbf{x}(t), \nvectwo_t \Big\rangle
+ \Big\langle \mathbf{x}'(t), \nvecone_t \Big\rangle.
$$
Combining this with the defining relation 
$$\Big\langle \mathbf{A}(t), \nvecone_t \Big\rangle = 
\Big\langle \mathbf{x}(t), \nvecone_t \Big\rangle + 1,
$$
we have two linear equations giving the orthonormal projections of $\mathbf{A}(t)$ in the directions $\nvecone_t$ and $\nvectwo_t$. It is immediate to see that the expression on the right-hand side of \eqref{eq:contact-pathA} satisfies those equations and hence is the correct expression for $\mathbf{A}(t)$. This proves \eqref{eq:contact-pathA}. The proof of the remaining equations \eqref{eq:contact-pathB}--\eqref{eq:contact-pathD} is similar and left to the reader.
\end{proof}

\paragraph{Example: generalized Hammersley sofas.}  Fix $0\le r\le 1$, and consider a semicircular rotation path $\mathbf{x}^{(r)}(t) = r(\cos(2t)-1,\sin(2t))$ of radius $r$ traveling from $(0,0)$ to $(-2r,0)$. An easy computation using \eqref{eq:contact-pathA}--\eqref{eq:contact-pathD} shows that the associated contact paths 
$\mathbf{A}^{(r)}(t)$, $\mathbf{B}^{(r)}(t)$, $\mathbf{C}^{(r)}(t)$, $\mathbf{D}^{(r)}(t)$ are given by
\begin{align*} 
\mathbf{A}^{(r)}(t) & =(\cos t, \sin t)^\top, \\
\mathbf{B}^{(r)}(t) & =(0,0)^\top, \\
\mathbf{C}^{(r)}(t) & =(-2r, 0)^\top, \\
\mathbf{D}^{(r)}(t) & =(-2r - \sin t, \cos t)^\top.
\end{align*}
Thus, the contact paths and $\mathbf{B}$ and $\mathbf{D}$ are fixed at the corners of the semicircular hole traced out by the rotation path, and the contact paths $\mathbf{A}$ and $\mathbf{C}$ trace out two quarter-circles of unit radius. (These assertions are also easy to prove using elementary geometry.) The sofa shape $S^{(r)}=S_{\mathbf{x}^{(r)}}$ therefore consists of two unit quarter-circles separated by a $2r\times 1$ rectangular block from which a semicircular piece of radius $r$ has been removed. 
This family of shapes was considered by Hammersley, who noticed that the area $f(r) = \frac{\pi}{2}+r\left(2-\frac{\pi}{2}r\right)$ of the shape takes its maximum value at $r_* = \frac{2}{\pi}$, this maximum being equal to $f(r_*) = \frac{\pi}{2}+\frac{2}{\pi}\approx2.2074$. The shape $S^{(r_*)}$ was the one proposed by Hammersley as a possible solution to the moving sofa problem.

\bigskip
In addition to our assumptions about the contact paths being well-defined, we introduce another geometric assumption about the rotation path~$\mathbf{x}$. For $0<t<\pi/2$, we say that $\mathbf{x}$ is \textbf{well-behaved} at $t$ if $\mathbf{x}$ is twice continuously differentiable at $t$ (which by Theorem~\ref{thm:contact-paths} also implies that any of the contact paths $\mathbf{A}(t)$, $\mathbf{B}(t)$, $\mathbf{C}(t)$ and $\mathbf{D}(t)$  that are defined at $t$ are continuously differentiable there), and if the following conditions hold:
\begin{enumerate}
\item If $\mathbf{x}(t)$ is a contact point then $\big\langle \mathbf{x}'(t), \nvectwo_t \big\rangle \ge 0$ and $\big\langle \mathbf{x}'(t), \nvecone_t \big\rangle \le 0$.
\item If $\mathbf{A}(t)$ is defined then $\big\langle \mathbf{A}'(t), \nvectwo_t \big\rangle \ge 0$.
\item If $\mathbf{B}(t)$ is defined then $\big\langle \mathbf{B}'(t), \nvectwo_t \big\rangle \le 0$.
\item If $\mathbf{C}(t)$ is defined then $\big\langle \mathbf{C}'(t), \nvecone_t \big\rangle \le 0$.
\item If $\mathbf{D}(t)$ is defined then $\big\langle \mathbf{D}'(t), \nvecone_t \big\rangle \ge 0$.
\end{enumerate}

It seems highly plausible that the rotation path associated with an area-maximizing shape will automatically be well-behaved at all except a finite number of values of $t$ where second differentiability fails, but we did not attempt to prove this. As we shall see in the next section, the assumption will prove useful in simplifying the form of certain equations.

\section{A family of six differential equations}

\label{sec:odes}

A key insight due to Gerver in his paper \cite{gerver} is that an area-maximizing shape in the moving sofa problem must be the limit of polygonal shapes satisfying a certain condition, which he referred to as \textbf{balanced polygons}. He defined a polygon to be balanced if, for any side of the polygon, that side and all other sides that are parallel to it lie on one of two lines, such that the distance between the lines is $1$ and the total lengths of the sides lying on each of the two lines are equal. By passing from the polygonal scenario to the limit of a curved shape he was able to derive his sofa shape.

While this was an important breakthrough, Gerver's computations were of a somewhat ad hoc nature and seem rather narrowly focused on his immediate goal of deriving his specific shape. In this section we extend his method to arrive at more general and explicit conditions that must hold for an area-maximizing moving sofa shape, and which we believe shed important new light on the problem. One change in our point of view is that we formulate our analogous ``balancedness'' condition for smooth shapes in terms of the rotation path $\mathbf{x}$, which as we showed in the previous section can be used to conveniently parametrize the associated shape $S_{\mathbf{x}}$.
More precisely, our result is a family of six ordinary differential equations that the rotation path must satisfy in different phases of the motion of the shape around the corner, as explained in the following theorem.

\begin{thm}
\label{thm:odes}
Let $\mathbf{x}$ be a rotation path, with an associated shape $S_{\mathbf{x}}$, set $\contact_{\mathbf{x}}$ of contact points, and contact paths $\mathbf{A}, \mathbf{B}, \mathbf{C}, \mathbf{D}$ as described in the previous section. Let $t\in(0,\pi/2)$ be a point where $\mathbf{x}(t)$ is 
well-behaved. Assume that $\contact=\contact_\mathbf{x}$ remains constant in a neighborhood of $t$ and is given by one of the six possibilities listed below. Then a necessary condition for the shape $S_\mathbf{x}$ to be a solution to the moving sofa problem is that $\mathbf{x}$ satisfies at $t$ one of the following six differential equations, according to the different possibilities for the set of contact points.

\bigskip \noindent $\bullet$
\textbf{Case 1:} $\contact_\mathbf{x}(t)= \{ \mathbf{A}, \mathbf{C}, \mathbf{D} \}$.
\begin{align*}
\mathbf{x}''(t) &=
\rotmat{t}
\Bigg(
\begin{pmatrix}-1 \\-1/2 \end{pmatrix}
+
\begin{pmatrix}
2\sin t & -2\cos t \\ 2\cos t & 2\sin t
\end{pmatrix} \mathbf{x}'(t) \Bigg)
\label{eq:ode1}
\tag{ODE1}
\end{align*}

\medskip \noindent $\bullet$
 \textbf{Case 2:} $\contact_\mathbf{x}(t)= \{ \mathbf{x}, \mathbf{A}, \mathbf{C}, \mathbf{D} \}$.
\begin{align*}\mathbf{x}''(t) &=
\rotmat{t}
\Bigg(
\begin{pmatrix}-1 \\-1/2 \end{pmatrix}
+
\begin{pmatrix}
\sin t & -\cos t \\ \frac32\cos t & \frac32\sin t
\end{pmatrix} \mathbf{x}'(t) \Bigg) \ 
\label{eq:ode2}
\tag{ODE2}
\end{align*}

\medskip \noindent $\bullet$
\textbf{Case 3:} $\contact_\mathbf{x}(t)= \{ \mathbf{x}, \mathbf{A}, \mathbf{C} \}$.
\begin{align*}
\mathbf{x}''(t) &=
\rotmat{t}
\Bigg(
\begin{pmatrix}-1 \\-1 \end{pmatrix}
+
\begin{pmatrix}
\sin t & -\cos t \\ \cos t & \sin t
\end{pmatrix} \mathbf{x}'(t)
\Bigg)
\qquad
\label{eq:ode3}
\tag{ODE3}
\end{align*}

\medskip \noindent $\bullet$
\textbf{Case 4:} $\contact_\mathbf{x}(t)= \{ \mathbf{x}, \mathbf{A}, \mathbf{B}, \mathbf{C} \}$.
\begin{align*}
\mathbf{x}''(t) &=
\rotmat{t}
\Bigg( \begin{pmatrix}-1/2 \\-1 \end{pmatrix}
+
\begin{pmatrix}
\frac32\sin t & -\frac32\cos t \\ \cos t & \sin t
\end{pmatrix} \mathbf{x}'(t) \Bigg)
\label{eq:ode4}
\tag{ODE4}
\end{align*}

\medskip \noindent $\bullet$
\textbf{Case 5:} $\contact_\mathbf{x}(t)= \{ \mathbf{A}, \mathbf{B}, \mathbf{C} \}$.
\begin{align*}
\mathbf{x}''(t) &=
\rotmat{t}
\Bigg(
\begin{pmatrix}-1/2 \\-1 \end{pmatrix}
+
\begin{pmatrix}
2\sin t & -2\cos t \\ 2\cos t & 2\sin t
\end{pmatrix} \mathbf{x}'(t) \Bigg)
\label{eq:ode5}
\tag{ODE5}
\end{align*}

\medskip \noindent $\bullet$
\textbf{Case 6:} $\contact_\mathbf{x}(t)= \{ \mathbf{x}, \mathbf{A}, \mathbf{B}, \mathbf{C}, \mathbf{D} \}$.
\begin{align*}
\mathbf{x}''(t) &=
\rotmat{t}
\Bigg(
\begin{pmatrix}-1/2 \\-1/2 \end{pmatrix}
+
\begin{pmatrix}
\frac32\sin t & -\frac32\cos t \\[3pt] \frac32\cos t & \frac32\sin t
\end{pmatrix} \mathbf{x}'(t) \Bigg)
\label{eq:ode6}
\tag{ODE6}
\end{align*}
Moreover, in Case~1 we have additionally that $\mathbf{A}'(t)=\mathbf{0}$ (that is, the contact point $\mathbf{A}$ remains fixed during an interval when Case~1 is applicable), and similarly in Case~5 we have that $\mathbf{C}'(t)=\mathbf{0}$.
\end{thm}

\begin{proof} 
Let us start by considering Case~3, which is conceptually the simplest of the six cases. We will derive \eqref{eq:ode3} by considering the effect on the area of the shape of certain local perturbations to the sequence of rigid motions encoded by the rotation path. This is a continuum version of the argument employed by Gerver, who used similar but less direct reasoning, starting from a discrete-geometric version of the moving sofa problem in which the intersection in \eqref{eq:rotation-shape} is assumed to take place over a finite set of values of $t$ (say, the values $t_k=k \pi/2n$, $k=0,\ldots,n$, where $n$ is a discrete parameter) and the area of the resulting shape is to be optimized over the resulting finite-dimensional configuration space of polygonal shapes. Our calculation works directly in the continuum regime and does not require taking a limit from the discrete variant of the problem.

Fix a small positive value $\delta$, and denote $t'=t+\delta$. The idea is to replace the rigid motion (simultaneous rotation and sliding) of the hallway $L$ encoded by the rotation path $\mathbf{x}$ by the following modified sequence of operations:
\begin{enumerate}
\item As a time coordinate $s$ ranges in $[0,t]$, drag the inner corner $(0,0)$ of the hallway along the rotation path $\mathbf{x}(s)$ while rotating the hallway around that corner (with the rotation angle being equal to the time coordinate $s$), similarly to the original sequence of motions.
\item Slide the hallway \emph{without rotating it} by a distance of $\delta$ in the direction of the vector $\nvecone_t$.
\item As $s$ ranges in $[t,t']$, drag the inner corner of the hallway, now positioned at $\mathbf{x}(t)+\delta \nvecone_t$, along the translated copy $\mathbf{x}(s) + \delta \nvecone_t$ of the segment $\mathbf{x}(s)$, $t\le s\le t'$ of $\mathbf{x}$, while continuing the rotation so that for each $s$ the angle of rotation is equal to $s$ as in the original motion.
\item Slide the hallway \emph{without rotating it} by a distance of $\delta$ in the direction of the vector $-\nvecone_{t'}$. The inner corner of the hallway is now at $\mathbf{x}(t')$.
\item Continue the rotation for $s\in [t',\pi/2]$ as prescribed by the original rotation path $\mathbf{x}$, similarly to step 1 above.
\end{enumerate}
Denote by $S_\mathbf{x}'$ the shape contained in the intersection of the hallway copies being translated and rotated in this modified sequence, formally given by the expression
\begin{align*}
S_\mathbf{x}' &=
\armh \cap \Big(\mathbf{x}(\pi/2)+\rotmat{\pi/2}(\armv) \Big)
\cap \bigcap_{0\le s\le t} \Big(\mathbf{x}(s)+\rotmat{s}(\hallway)\Big)
\\ & \qquad \cap \bigcap_{0\le r \le \delta} \Big(\mathbf{x}(t)+ r \nvecone_t+\rotmat{t}(\hallway)\Big)
\\ & \qquad \cap \bigcap_{t\le s\le t'} \Big(\mathbf{x}(s)+\delta \nvecone_t +\rotmat{s}(\hallway)\Big)
\\ & \qquad \cap \bigcap_{0\le r \le \delta} \Big(\mathbf{x}(t')+ r \nvecone_t+\rotmat{t'}(\hallway)\Big)
\\ & \qquad \cap \bigcap_{t'\le s\le \pi/2} \Big(\mathbf{x}(s)+\rotmat{s}(\hallway)\Big).
\end{align*}
Comparing $S_\mathbf{x}$ and $S_\mathbf{x}'$, we see that in changing from the former to the latter, some area was lost near the point $\mathbf{x}(t)$, and some area was gained near the contact point $\mathbf{A}(t)$. The part of the shape near the third contact point $\mathbf{C}$ remains the same, since for each $s\in [t,t']$, at time $s$ during step 3 of the motion, the outer hallway wall parallel to $\nvecone_s$ is tangent to the contact path $\mathbf{C}(s)$ just as in the original motion of the hallway.

Now, because of the assumption that $\mathbf{x}$ is differentiable at $t$, the shape of the piece that was lost is 
approximately a parallelogram with sides represented by the vectors $\delta \mathbf{x}'(t)$ and $\delta \nvecone_t$ incident to the point $\mathbf{x}(t)$, so its area is easily seen to be given (approximately for small $\delta$) by
\begin{equation}
\label{eq:area-lost}
\left|\Big\langle \mathbf{x}'(t), \nvectwo_t \Big\rangle \delta^2\right| + o(\delta^2) = 
\Big\langle \mathbf{x}'(t), \nvectwo_t \Big\rangle \delta^2 + o(\delta^2) \qquad (\delta \to 0),
\end{equation}
where the equality follows from the assumption that $\mathbf{x}$ is well-behaved at~$t$.
Similar reasoning applies to the piece that was gained, which is approximately a parallelogram (actually a rectangle, because $\mathbf{A}'(t)$ is parallel to $\nvectwo_t$) with sides represented by the vectors $\delta \mathbf{A}'(t)$ and $\delta \nvecone_t$ incident to the point $\mathbf{A}(t)$. Again using the well-behavedness assumption, the area of this rectangle is given for small $\delta$ by
\begin{equation}
\label{eq:area-gained}
 \Big\langle \mathbf{A}'(t), \nvectwo_t \Big\rangle \delta^2 + o(\delta^2) \qquad (\delta \to 0).
 \end{equation}
Comparing \eqref{eq:area-lost} and \eqref{eq:area-gained},  we see that under the assumption that the shape $S_\mathbf{x}$ has maximal area, the inequality
$$
\Big\langle \mathbf{x}'(t) - \mathbf{A}'(t), \nvectwo_t \Big\rangle  \ge 0
$$
must hold. But then, the reverse inequality $\langle \mathbf{x}'(t) - \mathbf{A}'(t), \nvectwo_t \rangle  \le 0$ must hold as well, because in the definition of the modified sequence of rigid motions we could have decided to push the hallway in the opposite direction $-\nvecone_t$ in step 2 and then in the direction $\nvecone_t$ in step 4, which would lead to a piece of area being \emph{gained}, instead of lost, near $\mathbf{x}(t)$, and a piece being lost instead of gained near $\mathbf{A}'(t)$, with the formulas \eqref{eq:area-lost} and \eqref{eq:area-gained} remaining valid but exchanging their meanings. Thus, we can conclude that under the area-maximization assumption, the rotation path must satisfy the differential 
relation
\begin{equation} \label{eq:differential-relation1}
\Big\langle \mathbf{x}'(t) - \mathbf{A}'(t), \nvectwo_t \Big\rangle  = 0.
\end{equation}
By a similar argument, it can now be shown that $\mathbf{x}$ satisfies a second differential relation, 
namely
\begin{equation} \label{eq:differential-relation2}
\Big\langle \mathbf{x}'(t) - \mathbf{C}'(t), \nvecone_t \Big\rangle  = 0.
\end{equation}
This is derived by considering a different modification of the sequence of rigid motions of the hallway, in which in steps 2 and 4 described above we slide it in a direction parallel to the vector $\nvectwo_t$ instead of $\nvecone_t$.

Now, recall from Theorem~\ref{thm:contact-paths} that $\mathbf{A}(t)$ and $\mathbf{C}(t)$ can be expressed in terms of $\mathbf{x}$. Differentiating \eqref{eq:contact-pathA} and \eqref{eq:contact-pathC} and using the relations $\nvecone_t'=\nvectwo_t$, $\nvectwo_t'=-\nvecone_t$, we get that
\begin{align*}
\mathbf{A}'(t) &= \mathbf{x}'(t) + \Big\langle \mathbf{x}''(t),\nvecone_t \Big\rangle \nvectwo_t + \Big\langle \mathbf{x}'(t), \nvectwo_t \Big\rangle \nvectwo_t - \Big\langle \mathbf{x}'(t), \nvecone_t \Big\rangle \nvecone_t + \nvectwo_t, \\
\mathbf{C}'(t) &= \mathbf{x}'(t) - \Big\langle \mathbf{x}''(t),\nvectwo_t \Big\rangle \nvecone_t + \Big\langle \mathbf{x}'(t), \nvecone_t \Big\rangle \nvecone_t - \Big\langle \mathbf{x}'(t), \nvectwo_t \Big\rangle \nvectwo_t - \nvecone_t.
\end{align*}
By substituting these expressions into the two equations 
\eqref{eq:differential-relation1}--\eqref{eq:differential-relation2} we get the pair of differential equations
\begin{align*}
\Big\langle \mathbf{x}''(t), \nvecone_t \Big\rangle &= -\Big\langle \mathbf{x}'(t), \nvectwo_t \Big\rangle -1, \\
\Big\langle \mathbf{x}''(t), \nvectwo_t \Big\rangle &= \Big\langle \mathbf{x}'(t), \nvecone_t \Big\rangle -1.
\end{align*}
It is now easily checked that \eqref{eq:ode3} is the same pair of equations written in matrix form.
This concludes our proof for Case~3.

Next, consider Case~2. Here we employ similar reasoning involving the same local modification of the sequence of rigid motions as described above, but now take into account the additional effect of perturbing the motion on the part of the shape near the contact point $\mathbf{D}(t)$. In this case the equation \eqref{eq:differential-relation1} is still satisfied, since the perturbed sequence of rigid motions that led to this equation does not change the shape near $\mathbf{D}(t)$, just like it did not affect the shape near $\mathbf{C}(t)$. However, in the second equation \eqref{eq:differential-relation2} an extra term needs to be introduced to take into account the behavior near $\mathbf{D}(t)$. By looking at the change in the areas between $S_\mathbf{x}$ and $S_\mathbf{x}'$ and reasoning as we did for Case~3, it is not hard to work out that the correct equation that should replace \eqref{eq:differential-relation2} is
\begin{equation} \label{eq:differential-relation3}
\Big\langle \mathbf{x}'(t) - \mathbf{C}'(t)  - \mathbf{D}'(t), \nvecone_t \Big\rangle  = 0.
\end{equation}
Again, taking \eqref{eq:differential-relation1} and \eqref{eq:differential-relation3} and substituting the expressions \eqref{eq:contact-pathC}--\eqref{eq:contact-pathD} for $\mathbf{C}(t)$ and $\mathbf{D}(t)$ yields \eqref{eq:ode2} after a short computation. This explains Case~2. Case~4 is completely analogous to Case~2 (in fact, they are related to each other by an obvious symmetry) and is handled similarly; one can check that in this case the two differential equations consist of \eqref{eq:differential-relation2} and a modified version of \eqref{eq:differential-relation1}, 
namely
\begin{equation} \label{eq:differential-relation4}
\Big\langle \mathbf{x}'(t) - \mathbf{A}'(t) - \mathbf{B}'(t), \nvectwo_t \Big\rangle  = 0.
\end{equation}
Again, a short computation, which we omit, brings this to the form of the vector ODE \eqref{eq:ode4}.

Next, Case~6 combines the two modifications of the equations for Case~3 that we made to handle Cases~2 and~4. Thus, the relevant pair of differential equations consists of \eqref{eq:differential-relation3} and \eqref{eq:differential-relation4}, which as before can be checked to be equivalent to \eqref{eq:ode6}.

It remains to consider Cases 1 and 5. Since they are symmetric to each other, we discuss only Case~1. This case can be thought of as a degenerate version of Case~2. The argument involving sliding the hallway in the direction of $\nvecone_t$ in step 2 and in the direction of $-\nvecone_{t'}$ in step 4 still applies, but results in the differential equation
\begin{equation} \label{eq:differential-relation5}
\Big\langle \mathbf{A}'(t), \nvectwo_t \Big\rangle  = 0
\end{equation}
instead of \eqref{eq:differential-relation1},
since $\mathbf{x}$ is not a contact point, so for small values of $\delta$ there is no change to the area near $\mathbf{x}(t)$. The second equation \eqref{eq:differential-relation3} from Case~2 is similarly replaced 
with
\begin{equation} \label{eq:differential-relation6}
\Big\langle \mathbf{C}'(t)  + \mathbf{D}'(t), \nvecone_t \Big\rangle  = 0.
\end{equation}
Once again, those two equations can be brought to the form of the single vector equation \eqref{eq:ode1} using routine algebra. Note also that because $\mathbf{A}'(t)$ is parallel to $\nvectwo_t$, \eqref{eq:differential-relation5} is equivalent to the relation $\mathbf{A}'(t)=\mathbf{0}$, which explains the additional claim in the theorem that the contact point $\mathbf{A}$ remains fixed on an interval in which Case~1 applies.
This completes the proof.
\end{proof}

The equations \eqref{eq:ode1}--\eqref{eq:ode6} are easy to solve, and their solutions will form the basis to our rederivation of Gerver's results and to our new construction in the ambidextrous moving sofa problem. We record the general form of these solutions in the following result.

\begin{thm}
The general solutions of the ODEs \eqref{eq:ode1}--\eqref{eq:ode6} are given respectively by
\begin{align}\mathbf{x}_1(t) &= 
\rotmat{t}
\begin{pmatrix}
a_1 \cos t + a_2 \sin t - 1 \\
-a_2 \cos t + a_1 \sin t - 1/2
\end{pmatrix} + \boldsymbol{\kappa}_1
\label{eq:ode1-gensolution}
\tag{SOL1}
\\[5pt]
\mathbf{x}_2(t) &= 
\rotmat{t}
\begin{pmatrix} 
-\tfrac14 t^2 + b_1 t + b_2 \\[3pt] \tfrac12 t - b_1 - 1
\end{pmatrix}
+\boldsymbol{\kappa}_2
\label{eq:ode2-gensolution}
\tag{SOL2}
\\[5pt]
\mathbf{x}_3(t) &= 
\rotmat{t}
\begin{pmatrix} 
c_1-t \\ c_2+t
\end{pmatrix}
+\boldsymbol{\kappa}_3
\label{eq:ode3-gensolution}
\tag{SOL3}
\\[5pt]
\mathbf{x}_4(t) &= 
\rotmat{t}
\begin{pmatrix} 
-\tfrac12 t + d_1 - 1 \\[3pt]
 -\tfrac14 t^2 + d_1 t + d_2
\end{pmatrix}
+\boldsymbol{\kappa}_4
\label{eq:ode4-gensolution}
\tag{SOL4}
\\[5pt]
\mathbf{x}_5(t) &= 
\rotmat{t}
\begin{pmatrix}
e_1 \cos t + e_2 \sin t - 1/2 \\
-e_2 \cos t + e_1 \sin t - 1
\end{pmatrix} + \boldsymbol{\kappa}_5
\label{eq:ode5-gensolution}
\tag{SOL5}
\\[5pt]
\mathbf{x}_6(t) &= 
\rotmat{t}
\begin{pmatrix}
f_1 \cos (t/2) + f_2 \sin (t/2) - 1 \\
-f_2 \cos (t/2) + f_1 \sin (t/2) - 1
\end{pmatrix} + \boldsymbol{\kappa}_6,
\label{eq:ode6-gensolution}
\tag{SOL6}
\end{align}
where $\boldsymbol{\kappa}_j=(\kappa_{j,1},\kappa_{j,2})^\top$ $(j=1,\ldots,6)$ and $a_i, b_i, c_i, d_i, e_i, f_i$ $(i=1,2)$ are arbitrary real constants.
\end{thm}

\begin{proof} 
By the general theory of linear ODEs it is enough to check that the equations are satisfied by the respective expressions. This is a routine computation, which we omit. See Section~5 of \texttt{MovingSofas}, the companion \texttt{Mathematica} package to this article \cite{romik-sofa-mathematica}, for an automated verification. Note that the solutions were derived by making the substitution $\mathbf{y}(t)=
\rotmat{-t}
\mathbf{x}'(t)$ and then rewriting each of the ODEs in terms of $\mathbf{y}(t)$. It is easy to check that with this substitution each of the six equations transforms into an equation of the form
$$
\mathbf{y}'(t) = T \mathbf{y} + \mathbf{v},
$$
where $T$ is a constant $2\times 2$ matrix and $\mathbf{v}$ is a constant vector. (For example, in the case of \eqref{eq:ode2} we get $T=\left(\begin{smallmatrix}0&0\\1/2&0 \end{smallmatrix}\right)$ and $\mathbf{v}=(-1,-1/2)^\top$, and in the case of \eqref{eq:ode3} 
we get $T=\left(\begin{smallmatrix}0&0\\0&0 \end{smallmatrix}\right)$ and $\mathbf{v}=(-1,-1)^\top$.)
The procedure for solving ODEs of this type is standard. 
\end{proof}

\section{A new derivation of Gerver's sofa}

\label{sec:gerver-rederivation}

We now show how the results of the previous sections can be used to give a transparent and conceptually simple derivation of Gerver's sofa. 
The idea is to look for a rotation path $\mathbf{x}$ that satisfies the following assumptions:

\begin{enumerate}

\item The path $\mathbf{x}$ is continuously differentiable.
\item The path $\mathbf{x}$ (and therefore also the associated shape $S_{\mathbf{x}}$) has a left-to-right symmetry around the vertical axis passing through its midpoint $\mathbf{x}(\pi/4)$.
\item The associated contact path $\mathbf{A}(t)$ satisfies 
\begin{equation} \label{eq:gerver-assumptionA}
\mathbf{A}(0)=(1,0)^\top.
\end{equation}
\item The set of contact points transitions through the following five distinct phases:
\begin{equation} \label{eq:gerver-contact-phases}
\contact_{\mathbf{x}}(t) = \begin{cases}
\{ \mathbf{A}, \mathbf{C}, \mathbf{D} \} & \textrm{if }\phantom{\pi/2-\ \,}0< t< \varphi, \\
\{ \mathbf{x}, \mathbf{A}, \mathbf{C}, \mathbf{D} \} & \textrm{if }\phantom{\pi/2-\ }\varphi \le t < \theta, \\
\{ \mathbf{x}, \mathbf{A}, \mathbf{C} \}& \textrm{if }\phantom{\pi/2-\ \,}\theta \le t \le \pi/2-\theta, \\
\{ \mathbf{x}, \mathbf{A}, \mathbf{B}, \mathbf{C} \}& \textrm{if }\, \pi/2-\theta < t \le \pi/2-\varphi, \\
\{ \mathbf{A}, \mathbf{B}, \mathbf{C} \} & \textrm{if } \pi/2-\varphi < t < \pi/2,
\end{cases}
\end{equation}
where $0<\varphi<\theta<\pi/4$ are two critical angles corresponding to where these transitions occur, and whose values need to be determined. 

\item During each of the five phases in \eqref{eq:gerver-contact-phases} the rotation path is well-behaved (in the sense of Section~\ref{sec:rotation-paths}).
\end{enumerate}

Under these assumptions, in order for the shape to be area-maximizing, the rotation path must satisfy in each of the phases of the rotation the correpsonding ODE from the family \eqref{eq:ode1}--\eqref{eq:ode5} given in Theorem~\ref{thm:odes}. (Note that the sixth differential equation \eqref{eq:ode6} does not appear; it plays no part in the derivation of Gerver's sofa, but will appear in our new construction for the ambidextrous moving sofa problem --- see Section~\ref{sec:ambidextrous}.) In other words, our rotation path must be of the form
\begin{equation}
\label{eq:gerverx-def}
\mathbf{x}(t) = \begin{cases}
\mathbf{x}_1(t) & \textrm{if }\phantom{\pi/2-\ \,}0\le t< \varphi, \\
\mathbf{x}_2(t) & \textrm{if }\phantom{\pi/2-\ }\varphi \le t < \theta, \\
\mathbf{x}_3(t) & \textrm{if }\phantom{\pi/2-\ \,}\theta \le t \le \pi/2-\theta, \\
\mathbf{x}_4(t) & \textrm{if }\, \pi/2-\theta < t \le \pi/2-\varphi, \\
\mathbf{x}_5(t) & \textrm{if }\pi/2-\varphi < t \le \pi/2.
\end{cases}
\end{equation}
obtained by gluing together the solutions \eqref{eq:ode1-gensolution}--\eqref{eq:ode5-gensolution} to the first five ODEs. The problem therefore reduces to the question of finding the values of the $22$ unknown parameters $\varphi$, $\theta$, $\boldsymbol{\kappa}_j$ $(j=1,\ldots,5)$, and $a_i, b_i, c_i, d_i, e_i$, $(i=1,2)$. But the parameters are not independent; rather, they satisfy a system of constraints arising out of the assumptions we made about the properties of $\mathbf{x}$. For example, the left-to-right symmetry condition can be expressed in the form of the relation
\begin{equation}
\mathbf{x}'(\pi/2-t) \equiv \begin{pmatrix} 1&0 \\ 0&-1 \end{pmatrix} \mathbf{x}'(t). \label{eq:gerver-symmetry}
\end{equation}
Referring to the definitions, it is easy to translate this into explicit linear relations between the parameters, namely the five equations
\begin{align}
e_1 &= a_1, \label{eq:gerver-coeffs-symmetry1} \\
e_2 &= -a_2, \label{eq:gerver-coeffs-symmetry2} \\
d_1 &= \frac{\pi}{4} - b_1, \label{eq:gerver-coeffs-symmetry3} \\
d_2 &= b_2 + \frac{\pi}{4}\left(2 b_1 - \frac{\pi}{4} \right), \label{eq:gerver-coeffs-symmetry4} \\
c_2 &= c_1-\frac{\pi}{2}. \label{eq:gerver-coeffs-symmetry5}
\end{align}
Next, the assumption that $\mathbf{A}(0)=(1,0)^\top$, together with the standard requirement that $\mathbf{x}(0)=(0,0)^\top$, translate to the three linear relations
$\kappa_{1,1} = 1-a_1$, $\kappa_{1,2} = 1/2-a_2$, $\kappa_{1,2} = -a_2$, which can be rewritten as
\begin{align}
\kappa_{1,1} &= 1-a_1, \label{eq:gerverxA-time0-exp1} \\
\kappa_{1,2} &= 1/4, \label{eq:gerverxA-time0-exp2} \\
a_2 &= -1/4. \label{eq:gerverxA-time0-exp3}
\end{align}
Next, the condition that $\mathbf{x}$ be continuously differentiable implies the vector relations
\begin{align}
\mathbf{x}_1(\varphi) &= \mathbf{x}_2(\varphi), \label{eq:gerver-cont12} \\
\mathbf{x}_1'(\varphi) &= \mathbf{x}_2'(\varphi), \label{eq:gerver-diff12} \\
\mathbf{x}_2(\theta) &= \mathbf{x}_3(\theta), \label{eq:gerver-cont23} \\
\mathbf{x}_2'(\theta) &= \mathbf{x}_3'(\theta), \label{eq:gerver-diff23} \\
\mathbf{x}_3(\pi/2-\theta) &= \mathbf{x}_4(\pi/2-\theta), \label{eq:gerver-cont34} \\
\mathbf{x}_3'(\pi/2-\theta) &= \mathbf{x}_4'(\pi/2-\theta), \label{eq:gerver-diff34} \\
\mathbf{x}_4(\pi/2-\varphi) &= \mathbf{x}_5(\pi/2-\varphi), \label{eq:gerver-cont45} \\
\mathbf{x}_4'(\pi/2-\varphi) &= \mathbf{x}_5'(\pi/2-\varphi), \label{eq:gerver-diff45}
\end{align}
of which \eqref{eq:gerver-diff34} and \eqref{eq:gerver-diff45} are redundant, since they are easily seen to follow automatically from \eqref{eq:gerver-cont12}--\eqref{eq:gerver-diff23} together with the symmetry assumptions.

Finally, we have two additional vector equations,
\begin{align}
\mathbf{x}_1(\varphi) &= \mathbf{B}(\pi/2-\theta), \label{eq:gerver-contactintersect1} \\
\mathbf{x}_5(\pi/2-\varphi) &= \mathbf{D}(\theta), \label{eq:gerver-contactintersect2}
\end{align}
that encapsulate the requirement that the transitions described in \eqref{eq:gerver-contact-phases} between the different phases for the set of contact points  occur where we assumed they do. Here, too, the second equation \eqref{eq:gerver-contactintersect2} is redundant and follows from \eqref{eq:gerver-contactintersect1} and the symmetry assumption.

The equations 
\eqref{eq:gerver-coeffs-symmetry1}--\eqref{eq:gerver-contactintersect2} 
comprise a total of 28 (scalar) equations in the 22 variables $\varphi$, $\theta$, $\kappa_j$, $a_i,b_i,c_i,d_i,e_i$, of which 6 were noted as being redundant, leaving 22 truly independent equations, equal to the number of variables. It is not immediately obvious, but this system of equations has a unique solution. Moreover, finding the numerical values of the parameters is now a simple matter of entering the equations into \texttt{Mathematica} and invoking its \texttt{FindRoot[$\cdots$]} command to numerically solve the system. This immediately yields the desired numerical values to any reasonable desired level of accuracy. The computation is carried out in \texttt{MovingSofas}, the companion \texttt{Mathematica} package to this article \cite[Section~6]{romik-sofa-mathematica}. The numerical values of the parameters are listed in Table~\ref{table:gerver-numerical-constants} for reference.

\begin{table}
$$
\begin{array}{|cl|cl|} \hline
\varphi & \phantom{-}0.039 177 364 790 083 641\ldots & 
\theta & \phantom{-}0.681 301 509 382 724 894\ldots 
\\[5pt]
\kappa_{1,1} & -0.210 322 422 072 688 751 \ldots & a_1 & \phantom{-}1.210 322 422 072 688 751\ldots
\\
\kappa_{1,2} & \phantom{-}1/4  & a_2 & -1/4
\\[5pt]
\kappa_{2,1} & -0.919 179 292 771 593 322 \ldots & b_1 & -0.527 624 598 026 784 624\ldots
\\
\kappa_{2,2} & \phantom{-}0.472 406 619 750 805 465 \ldots & b_2 & \phantom{-}0.920 258 385 160 637 622\ldots
\\[5pt]
\kappa_{3,1} & -0.613 763 229 430 251 668\ldots & c_1 & \phantom{-}0.626 045 522 848 465 867\ldots
\\
\kappa_{3,2} & \phantom{-}0.889 626 479 003 221 860\ldots & c_2 & -0.944 750 803 946 430 751\ldots
\\[5pt]
\kappa_{4,1} & -0.308 347 166 088 910 014\ldots & d_1 & \phantom{-}1.313 022 761 424 232 933\ldots
\\
\kappa_{4,2} & \phantom{-}0.472 406 619 750 805 465\ldots & d_2 & -0.525 382 670 414 554 437\ldots
\\[5pt]
\kappa_{5,1} & -1.017 204 036 787 814 585\ldots & e_1 & \phantom{-}1.210 322 422 072 688 751\ldots
\\
\kappa_{5,2} & \phantom{-}1/4 & e_2 & \phantom{-}1/4
\\ \hline 
\end{array}
$$
\caption{Numerical values for the constants in Gerver's sofa.}
\label{table:gerver-numerical-constants}
\end{table}

While the technique described above provides the quickest and most effortless way to get the value of the constants, it is worth taking a closer look at the system of equations we are solving to get a better insight into its structure, which may be useful, for example, if one wishes to prove that the solution is unique, and in preparation for our analysis of the ambidextrous moving sofa problem in the next section. A key observation is that all our equations are linear in all the parameters $\kappa_{j,1}, \kappa_{j,2}$, $a_i,b_i,c_i,d_i,e_i$, and are only nonlinear in the two critical angles $\varphi, \theta$. 
This suggests that a large part of the solution to the system can be carried out symbolically, with only the last step involving a numerical root-finding procedure. Thus, an alternative approach to solving the system is to pick a set of $20$ of the equations (out of the $22$ we used in the purely numerical approach described above); solve it as a linear system in the $20$ ``linear'' parameters to obtain symbolic expressions for those $20$ parameters in terms of the two angular variables $\varphi$ and $\theta$; then substitute those expressions into the two remaining equations, to obtain two nonlinear equations in $\varphi$ and $\theta$, which can then be solved numerically. Our companion \texttt{Mathematica} package illustrates this method as well \cite[Section~7]{romik-sofa-mathematica}, and also shows how to compute the area $2.21953166\ldots$ of Gerver's sofa to high accuracy \cite[Section~8]{romik-sofa-mathematica}.  

\section{An exact solution in the ambidextrous moving sofa problem}

\label{sec:ambidextrous}

Having rederived Gerver's shape, we now show how the same techniques can be used with slight modification to derive a new shape that, analogously to Gerver's sofa, is a highly plausible candidate to be the solution to the ambidextrous moving sofa problem.

The idea behind our new construction is to look for a rotation path $\mathbf{x}$, with an associated shape $S_{\mathbf{x}}$, that satisfies a modified version of the list of assumptions in our derivation of Gerver's shape. Namely, we assume:

\begin{enumerate}

\item The path $\mathbf{x}$ is continuously differentiable, as before.
\item The path $\mathbf{x}$ has a left-to-right symmetry, as before.
\item The associated contact path $\mathbf{A}(t)$ satisfies 
\begin{equation} \label{eq:ambi-assumptionA}
\mathbf{A}(0)=(1,1/2)^\top
\end{equation}
(compare with \eqref{eq:gerver-assumptionA}).
\item The set of contact points transitions through the following three (instead of five) distinct phases:
\begin{equation} \label{eq:ambi-contact-phases}
\contact_{\mathbf{x}}(t) = \begin{cases}
\{ \mathbf{A}, \mathbf{C}, \mathbf{D} \} & \textrm{if }\phantom{\pi/2-\ \,}0< t< \beta, \\
\{ \mathbf{x}, \mathbf{A}, \mathbf{B}, \mathbf{C}, \mathbf{D} \}& \textrm{if }\phantom{\pi/2-\ \,}\beta \le t \le \pi/2-\beta, \\
\{ \mathbf{A}, \mathbf{B}, \mathbf{C} \} & \textrm{if } \pi/2-\beta < t < \pi/2,
\end{cases}
\end{equation}
(compare with \eqref{eq:gerver-contact-phases}) where $0<\beta<\pi/4$ is a new critical angle whose value needs to be determined. 

\item During each of the five phases in \eqref{eq:ambi-contact-phases} the rotation path is well-behaved.

\end{enumerate}

Now observe that given any rotation path $\mathbf{x}$ and an associated shape $S_\mathbf{x}$, one can trivially turn the shape into an ``ambidextrous shape'' that can move around corners both to the left and to the right by replacing it with its intersection with its reflection across the line $y=1/2$; see Fig.~\ref{fig:ambidextrous-reflection}, which illustrates why the assumption \eqref{eq:ambi-assumptionA} is precisely the condition that makes the most efficient use of this type of symmetrization in terms of maximizing the area. 
\begin{figure}
\begin{center}
\scalebox{0.6}{\includegraphics{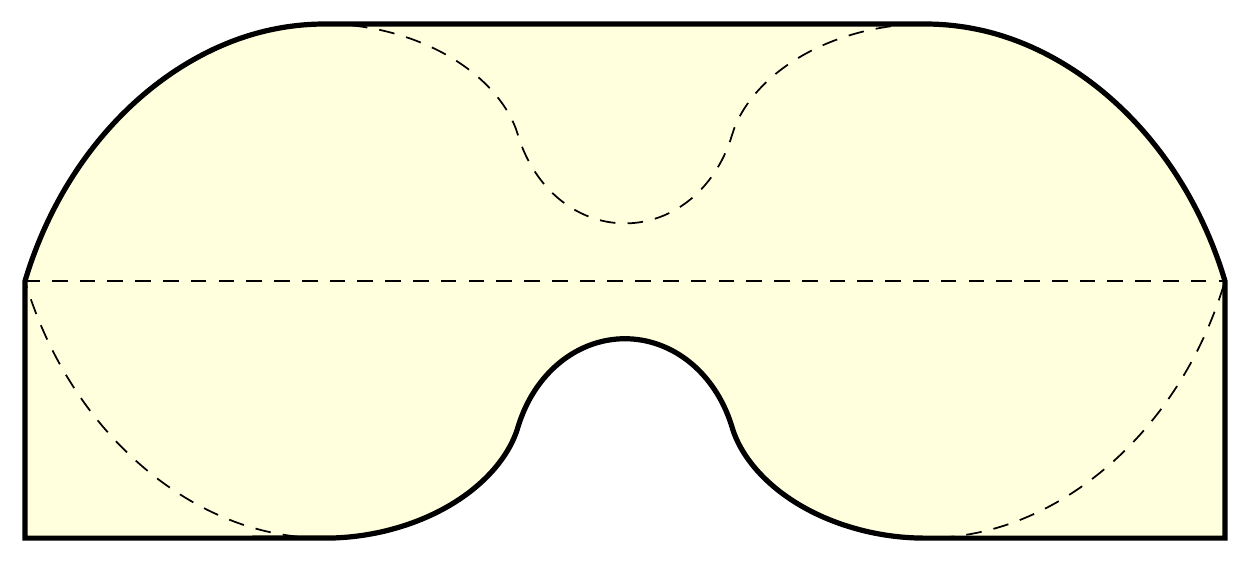}}
\caption{The shape $S_{\mathbf{x}}$ is turned into an ambidextrous shape by intersecting it with its  reflection across the line $y=1/2$.} \label{fig:ambidextrous-reflection}
\end{center}
\end{figure}
Furthermore, in order for such an up-down-symmetrized shape to have maximal area for the ambidextrous moving sofa problem, it should certainly be a \emph{local maximum} of the area, and in particular its area should only decrease under the kinds of local perturbations that were used in the proof of Theorem~\ref{thm:odes}. Thus, we see that the necessary conditions of Theorem~\ref{thm:odes} would need to hold, and the above assumptions together with the area-maximization assumption therefore imply that
the rotation path $\mathbf{x}$ must be of the form
\begin{equation} \label{eq:ambi-rotation-path}
\mathbf{x}(t) = \begin{cases}
\mathbf{x}_1(t) & \textrm{if }\phantom{\pi/2-\ \,}0< t< \beta, \\
\mathbf{x}_6(t) & \textrm{if }\phantom{\pi/2-\ \,}\beta \le t \le \pi/2-\beta, \\
\mathbf{x}_5(t) & \textrm{if } \pi/2-\beta < t < \pi/2,
\end{cases}
\end{equation}
with $\mathbf{x}_1(t)$, $\mathbf{x}_6(t)$ and $\mathbf{x}_5(t)$ being given by \eqref{eq:ode1-gensolution}, \eqref{eq:ode6-gensolution} and \eqref{eq:ode5-gensolution}.
As with our derivation of Gerver's sofa, the goal is now to find values of the parameters
$\beta$, $\boldsymbol{\kappa}_j$ $(j=1,6,5)$, and $a_i, f_i, e_i$, $(i=1,2)$ that enter into the definition of $\mathbf{x}$ and for which the above assumptions on $\mathbf{x}$ are satisfied. 

To proceed, we translate the list of assumptions on $\mathbf{x}$ into a concrete system of equations, following similar reasoning to that employed in Section~\ref{sec:gerver-rederivation}. First, the symmetry condition (encapsulated by \eqref{eq:gerver-symmetry}) now translates to the three equations
\begin{align}
e_1 &= a_1, \label{eq:ambi-coeffs-symmetry1} \\
e_2 &= -a_2, \label{eq:ambi-coeffs-symmetry2} \\
f_2 &= (1-\sqrt{2})f_1. \label{eq:ambi-coeffs-symmetry3}
\end{align}
Second, it is readily checked that the equations $\mathbf{x}(0)=(0,0)^\top$, $\mathbf{A}(0)=(1,1/2)$ are equivalent to the linear relations
\begin{align}
\kappa_{1,1}&=1-a_1, \label{eq:ambi-xA-relations1} \\
a_2 &= 0, \label{eq:ambi-xA-relations2}
\\
\kappa_{1,2} &= 1/2. \label{eq:ambi-xA-relations3}
\end{align}
Third, the assumption that the rotation path is continuously differentiable translates to the equations
\begin{align}
\mathbf{x}_1(\beta) &= \mathbf{x}_6(\beta), \label{eq:ambi-cont12} \\
\mathbf{x}_1'(\beta) &= \mathbf{x}_6'(\beta), \label{eq:ambi-diff12} \\
\mathbf{x}_6(\pi/2-\beta) &= \mathbf{x}_5(\pi/2-\beta), \label{eq:ambi-cont34} \\
\mathbf{x}_6'(\pi/2-\beta) &= \mathbf{x}_5'(\pi/2-\beta). \label{eq:ambi-diff34} 
\end{align}
Here, the last equation \eqref{eq:ambi-diff34} is redundant and follows from \eqref{eq:ambi-diff12} together with the symmetry relations \eqref{eq:ambi-coeffs-symmetry1}--\eqref{eq:ambi-coeffs-symmetry3}.

Finally, the assumption \eqref{eq:ambi-contact-phases} regarding the structure of the set of contact points will be satisfied if the two equations
\begin{align}
\mathbf{x}_1(\beta) &= \mathbf{B}(\beta), \label{eq:ambi-contactintersect1} \\
\mathbf{x}_5(\pi/2-\varphi) &= \mathbf{D}(\pi/2-\beta) \label{eq:ambi-contactintersect2}
\end{align}
hold. In this pair of equations \eqref{eq:ambi-contactintersect2} is again redundant and follows from \eqref{eq:ambi-contactintersect1} and symmetry. Furthermore, by \eqref{eq:contact-pathB}, the vector equation \eqref{eq:ambi-contactintersect1} is actually equivalent to the scalar equation
\begin{equation}
\big\langle \mathbf{x}_1'(\beta), \nvecone_t \big\rangle = 0. \label{eq:ambi-contactintersect3}
\end{equation}

The procedure for solving the equations is now very similar to the one we carried out in the previous section, with the crucial difference that the equations in this case are solvable in closed form.

\begin{thm}
\label{thm:ambi-uniquesolution}
The equations \eqref{eq:ambi-coeffs-symmetry1}--\eqref{eq:ambi-contactintersect3} have a unique solution in the $13$ parameters $\beta$, $\kappa_{1,i}, \kappa_{6,i}, \kappa_{5,i}, a_i, f_i, e_i$ ($i=1,2$), which are given by the following expressions:
\begin{align}
\beta &=
  \arctan\left[
\frac{1}{2} \left( \sqrt[3]{\sqrt{2}+1}- \sqrt[3]{\sqrt{2}-1}\, \right)
  \right],
\label{eq:ambi-symbolic-beta} \\
a_1 &= e_1 = \frac14 \operatorname{cosec} \beta = \frac{1}{4} \sqrt{4+\sqrt[3]{71+8 \sqrt{2}}+\sqrt[3]{71-8 \sqrt{2}}}, 
\\
a_2 &= e_2 = 0, \\
\kappa_{1,2} &= \kappa_{6,2} = \kappa_{5,2} = \frac12,
\label{eq:ambi-k62}
\\
\kappa_{1,1} &= 1-a_1,
\\
\kappa_{6,1} &=1-\frac43 a_1,
\label{eq:ambi-k61}
\\
\kappa_{5,1} &= 1-\frac53 a_1,
\\
f_1 &= 
\frac{\Big(83+\sqrt[3]{420619+15104 \sqrt{2}}+\sqrt[3]{420619-15104 \sqrt{2}}\Big)^{1/4}}{3 \sqrt{2 \left(2-\sqrt{2}\right)}}, 
\label{eq:ambi-f1}
\\
f_2 &=
 -(\sqrt{2}-1)f_1,
 \label{eq:ambi-f2}
\end{align}
The numerical values of these parameters are listed in Table~\ref{table:ambi-numerical-constants}. 
\end{thm}

\begin{table}
$$
\begin{array}{|cl|cl|}
\hline
\beta & \phantom{-}0.289653820817320941\ldots  \ \ & \ \ & 
\\[5pt]
\kappa_{1,1} & \phantom{-}0.124712637587267758\ldots  & a_1 & \phantom{-}0.875287362412732241\ldots
\\
\kappa_{1,2} & \phantom{-}1/2 & a_2 & \phantom{-}0
\\[5pt]
\kappa_{6,1} & -0.167049816550309655  \ldots & f_1 & \phantom{-}1.202 938 908 156 911 389\ldots
\\
\kappa_{6,2} & \phantom{-}1/2  & f_2 & -0.498 273 610 464 875 672 \ldots
\\[5pt]
\kappa_{5,1} & -0.458812270687887068\ldots & e_1 & 
\phantom{-}0.875 287 362 412 732 241\ldots
\\
\kappa_{5,2} & \phantom{-}1/2 & e_2 & 
\phantom{-}0 
\\[2pt] \hline
\end{array}
$$
\caption{Numerical values for the constants in our construction.}
\label{table:ambi-numerical-constants}
\end{table}

\begin{proof}
The system of equations we wrote down consists of precisely $13$ independent scalar equations in the $13$ variables, namely \eqref{eq:ambi-coeffs-symmetry1}--\eqref{eq:ambi-cont34} and \eqref{eq:ambi-contactintersect3}, together with four additional equations that were pointed out to be redundant. Moreover, the equations are all linear in the $12$ parameters $\kappa_{1,i}, \kappa_{6,i}, \kappa_{5,i}, a_i, f_i, e_i$ ($i=1,2$). Using \texttt{Mathematica}, we can solve the linear system consisting of the first $12$ equations of the $13$, to get expressions for these linear parameters as functions of $\beta$. Substituting these expressions back into the remaining equation gives a single nonlinear equation for $\beta$, which upon simplification becomes the relation
$$
3 \sin \left(\frac{\beta }{2}\right)+\sin \left(\frac{3 \beta }{2}\right)+ \left(\sqrt{2}-1\right) \left( -3\cos \left(\frac{\beta }{2}\right)+\cos \left(\frac{3 \beta }{2}\right) \right)=0.
$$
\textit{Crucially, this equation is algebraic in
$Z=\tan(\beta)$};
it can be easily solved, to give that
$$ 
\tan(\beta)=
\frac{1}{2} \left( \sqrt[3]{\sqrt{2}+1}- \sqrt[3]{\sqrt{2}-1}\, \right).
$$
This proves that the system has a unique solution (under the assumption $0<\beta<\pi/4$) and establishes the relation \eqref{eq:ambi-symbolic-beta}. Once $\beta$ is found, its value can be substituted into the formulas for the other $12$ parameters, which are all rational functions in the algebraic numbers $\cos(\beta/2)=\big(\frac12+\frac{1}{2\sqrt{1+Z^2}}\big)^{1/2}$ and $\sin(\beta/2)=\big(\frac12-\frac{1}{2\sqrt{1+Z^2}}\big)^{1/2}$. This shows that these parameters are algebraic numbers as well. Routine algebraic computations, which would be tedious to do by hand but can be performed automatically in \texttt{Mathematica} or other symbolic math applications, can now be used to verify the correctness of the formulas listed in the theorem. The details are found in the \texttt{MovingSofas} companion software package \cite[Section~10]{romik-sofa-mathematica}.
\end{proof}

We summarize our findings with the following theorem, which is the main result of the paper.

\begin{thm} Let $\mathbf{x}$ be the rotation path \eqref{eq:ambi-rotation-path}, with the associated numerical parameters being given by \eqref{eq:ambi-symbolic-beta}--\eqref{eq:ambi-f2}. Denote $\Sigma=S_\mathbf{x}\cap \rho(S_\mathbf{x})$, where $\rho$ is the affine reflection in the plane across the line $y=1/2$. Then $\Sigma$ is a shape that can move around corners both to the left and to the right. The rotation path~$\mathbf{x}$ is the unique one 
satisfying the assumptions 1--5 stated at the beginning of this section and that satisfies the necessary conditions from Theorem~\ref{thm:odes} at all $t\in (0,\pi/2)\setminus\{\beta,\pi/2-\beta\}$. The area $\Delta$ of the shape $\Sigma$ is given by \eqref{eq:ambi-area-closedform}, and the distance $\lambda$ between the left and right endpoints of $\Sigma$ is given by
\eqref{eq:ambi-length-closedform}.
\end{thm}

\begin{proof} We have already explained all the claims, except the computation of the area $\Delta$ and the length $\lambda$ of the shape. 
For the length, note that the coordinates of the left and right endpoints of $\Sigma$ are $\mathbf{C}(\pi/2)=(C_1(\pi/2),1/2)$ and $\mathbf{A}(0)=(1,1/2)$, so we have that
\begin{align*}
\lambda &= 1-C_1(\pi/2) = 1+a_1-\kappa_{5,1}=1+a_1-\left(1-\frac53 a_1\right)
\\ &=  \frac83 a_1 = \frac{2}{3} \sqrt{4+\sqrt[3]{71+8 \sqrt{2}}+\sqrt[3]{71-8 \sqrt{2}}},
\end{align*}
as claimed. Regarding the area, from the left-right and up-down symmetry of the shape we see that $\Delta$ can be expressed as the sum of integrals
\begin{align*}
\Delta &= 4\Bigg[\int_\beta^{\pi/2} (1/2-A_1(t))A_2'(t)\,dt
\\ & \qquad\qquad
+ \int_\beta^{\pi/2} (1/2-B_1(t))B_2'(t)\,dt  +
\int_\beta^{\pi/4} (x_1(t)-1/2)x_2'(t)\,dt\Bigg],
\end{align*}
where we denote $\mathbf{x}(t)=(x_1(t),x_2(t))^\top$, $\mathbf{A}(t)=(A_1(t),A_2(t))^\top$, $\mathbf{B}(t)=(B_1(t),B_2(t))^\top$. The integration can be performed symbolically and the result simplified by \texttt{Mathematica}
to show that indeed
$$
\Delta = \sqrt[3]{2 \sqrt{2}+3}+\sqrt[3]{3-2 \sqrt{2}}-1 + \beta.
$$
See \cite[Section~11]{romik-sofa-mathematica} for the details.
\end{proof}

Our calculations involved several curious algebraic numbers. We expressed them in radicals, but of course they can be alternatively (and perhaps better) described in terms of their minimal polynomials. The minimal polynomials, computed again using \texttt{Mathematica} \cite[Section~12]{romik-sofa-mathematica}, are listed in Table~\ref{table:minpolys}.

\begin{table}
$$
\!\!\!
\begin{array}{|cl|}
\hline
\textrm{Quantity} & \textrm{Minimal polynomial} \\
\hline
\vphantom{1^{1^{1^{1^1}}}}
\tan(\beta) & 4x^3+3x-1
\\[3pt]
\sin(\beta) & 2 x^6+3 x^4+12 x^2-1
\\[3pt]
\cos(\beta) & 2 x^6-9 x^4+24 x^2-16 
\\[3pt]
a_1, e_1 & 2048 x^6-1536 x^4-24 x^2-1
\\[3pt]
f_1, f_2 & 
{\scriptstyle 4251528 x^{12}-9920232 x^{10}+6672537 x^8-1936224 x^6+256608 x^4-13824 x^2+256}
\\[3pt]
\kappa_{1,1} & 2048 x^6-12288 x^5+29184 x^4-34816 x^3+21480 x^2-6096 x+487
\\[3pt]
\kappa_{6,1} & 729 x^6-4374 x^5+9963 x^4-10692 x^3+5076 x^2-432 x-272
\\[3pt]
\kappa_{5,1} & 
{\scriptstyle 
1492992 x^6-8957952 x^5+19284480 x^4-17418240 x^3+3597480 x^2+3753648 x-1768033
}
\\[3pt]
\Delta-\beta & x^3+3 x^2-8
\\[3pt]
\lambda & 729 x^6-3888 x^4-432 x^2-128
\\[2pt]
\hline
\end{array}
$$
\caption{The minimal polynomials of some of the algebraic numbers associated with the shape $\Sigma$.}
\label{table:minpolys}
\end{table}

\newpage

\section{Geometric and algebraic properties of the shape~$\Sigma$}

\label{sec:geom-alg-properties}

The shape $\Sigma$ we constructed seems like quite a natural and symmetric object. Moreover, in addition to its pleasing analytic and algebraic properties discussed so far, the shape exhibits several additional interesting geometric and algebraic relationships. The geometric properties are shown in Fig.~\ref{fig:ambi-annotated}; note the multiple appearances of the critical angle $\beta$, and the fact that the segments $\sigma_2, \sigma_3, \sigma_{16}, \sigma_{17}$ and $\sigma_7, \sigma_8, \sigma_{11}, \sigma_{12}$ of the boundary of the shape are circular arcs lying on the circles of radius $1/2$ around the ``focal points'' $F_2$ and $F_1$, respectively (which in particular means that the pairs of segments $\sigma_j, \sigma_{j+1}$ for each of $j=2,7,11,16$ could in principle be considered as single analytical segments, reducing the total number of segments involved in the description of the shape from $18$ to $14$; we chose however to describe the segments in each of these pairs separately, since they arise out of separate analytical processes and so far as we know it is only by computation that one can verify they belong to the same analytic curve). Another interesting geometric property, which can be easily verified from our formulas but for which we can see no obvious geometric explanation, is that the distance between the two focal points is precisely half the total length of the shape.

\begin{figure}
\begin{center}
\scalebox{0.6}{\includegraphics{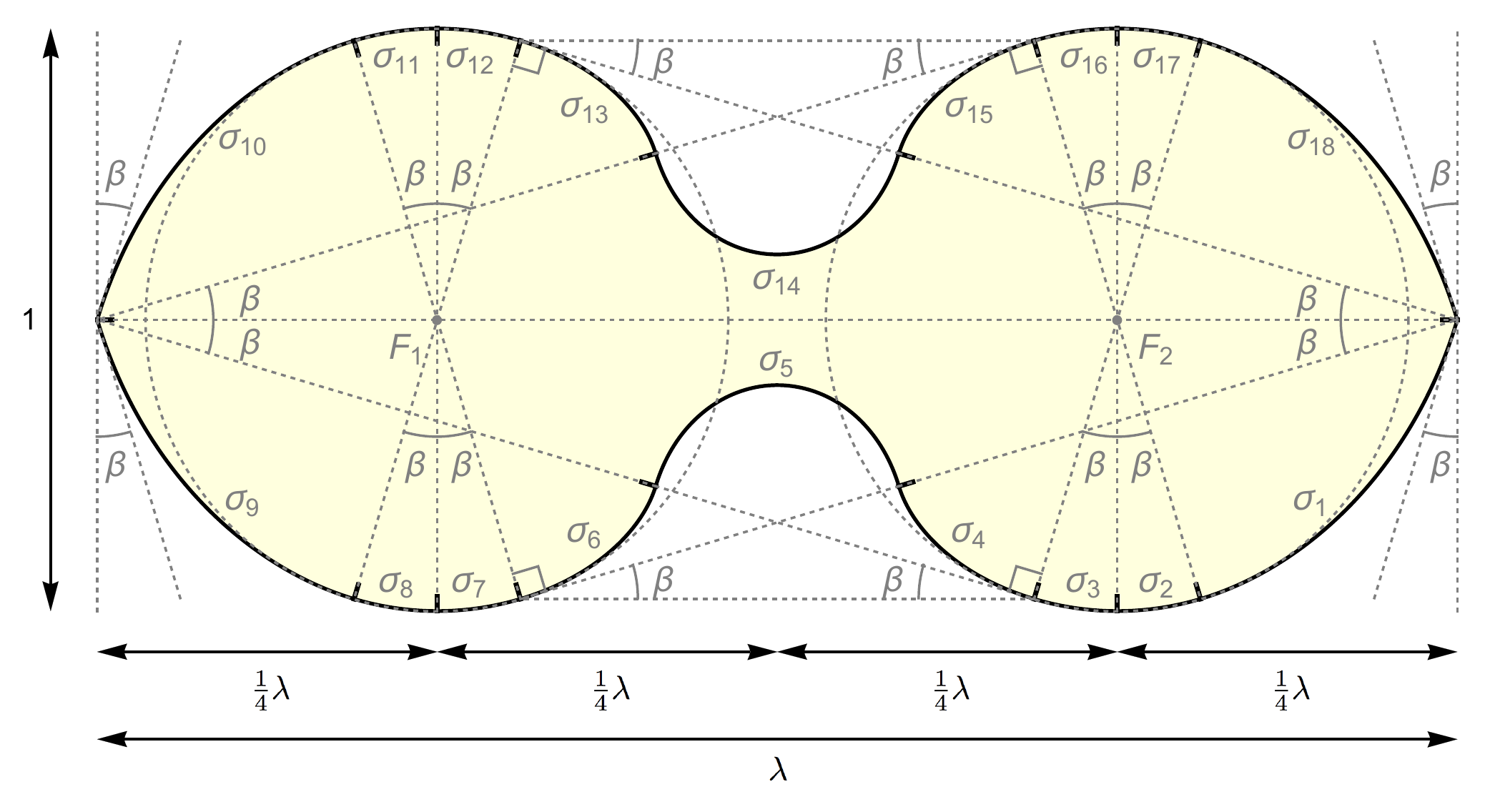}}
\caption{The shape $\Sigma$ with annotations highlighting interesting geometric relationships.}
\label{fig:ambi-annotated}
\end{center}
\end{figure}

Turning to algebraic properties of $\Sigma$, we have the feature that the $18$ segments of the boundary all lie on algebraic curves, as was mentioned in the introduction. A simple example of this are the circular arcs $\sigma_2, \sigma_3, \sigma_{16}, \sigma_{17}$, $\sigma_7, \sigma_8, \sigma_{11}, \sigma_{12}$ (in the notation of Fig.~\ref{fig:ambi-annotated}) already pointed out above. More intriguingly, there are three other distinct algebraic curves (up to obvious symmetries) of degree $6$ that appear as analytic continuations of boundary segments. To write down the equations for these curves, it is convenient to switch to new coordinates $X, Y$ defined through the affine change of variables
$$
X = \frac{x-\kappa_{6,1}}{\tfrac14 \sqrt{2-\sqrt{2}} f_1}, 
\qquad
Y = \frac{y-\kappa_{6,2}}{\tfrac14 \sqrt{2-\sqrt{2}} f_1}
$$
from the coordinates $x,y$ used in our original description of the shape $\Sigma$
(where $f_1$, $\kappa_{6,1}$, $\kappa_{6,2}$ are the values given in Theorem~\ref{thm:ambi-uniquesolution}). In these new coordinates, it can be shown (see the companion \texttt{Mathematica} package \cite[Section~13]{romik-sofa-mathematica}) that the boundary segments $\sigma_{18}$ and $\sigma_9$ of $\Sigma$ both lie on the algebraic curve
$$
P(X,Y) = 0,
$$
and the segments $\sigma_4$ and $\sigma_5$ lie on the algebraic curves
\begin{align*}
Q(X,Y) & = 0, \\ 
R(X,Y) & = 0,
\end{align*}
respectively, where $P(X,Y)$, $Q(X,Y)$, $R(X,Y)$ are polynomials defined by
\begin{align*}
P(X,Y) &= 
\left(X^2+Y^2-8\right)^3-216 (Y-X)^2,
\\[5pt]
Q(X,Y) &= 
\left(X^2+Y^2\right)^3
-12 \gamma _1 \left(X^2+Y^2\right)^2
-216 \sqrt{\gamma _2} \left(X^2+Y^2\right) (Y-X)
\\ & \qquad -12 \gamma _3 \left(X^2+Y^2\right)
-432 \sqrt{\gamma _4} (Y-X)
+432 X Y -32 \gamma _5,
\\[5pt]
R(X,Y) &= 
\left(X^2+Y^2\right)^3
-24 \alpha _1 \left(X^2+Y^2\right)^2
+48 \alpha _2 \left(X^2+Y^2\right)
\\ & \qquad +13824 \sqrt{\alpha _3} Y
+4096 \alpha _4.
\end{align*}
Here, $\gamma_1, \gamma_2, \gamma_3, \gamma_4, \gamma_5$, $\alpha_1, \alpha_2, \alpha_3, \alpha_4$ are cubic algebraic numbers, which can be expressed in terms of the constant $Z= (4+2\sqrt{2})^{1/3} + (4-2\sqrt{2})^{1/3}$ in the form
\begin{equation*}
\begin{array}{rlcrl}
\gamma_1 &= \displaystyle -3Z+14, 
& \qquad & 
\alpha_1 &= \displaystyle -3Z+16,
\\[3pt]
\gamma_2 &= \displaystyle -Z+4,
& &
\alpha_2 &= \displaystyle 27 Z^2 -240 Z + 592,
\\[3pt]
\gamma_3 &= \displaystyle -27 Z^2 + 156 Z - 190,
& &
\alpha_3 &= \displaystyle 12Z^2 -54Z + 56,
\\[3pt]
\gamma_4 &= \displaystyle 8Z^2 -26Z + 8,
& &
\alpha_4 &= \displaystyle -9Z + 28.
\\[3pt]
\gamma_5 &= \displaystyle 9Z-20, 
\end{array}
\end{equation*}
(This way of representing the constants $\gamma_i, \alpha_j$ was suggested to us by Greg Kuperberg, who also found the above way of expressing the polynomial $P(X,Y)$ that simplifies our earlier formula.)


\begin{figure}
\begin{center}
\begin{tabular}{cc}
\hspace{-30pt}
\scalebox{0.65}{\includegraphics{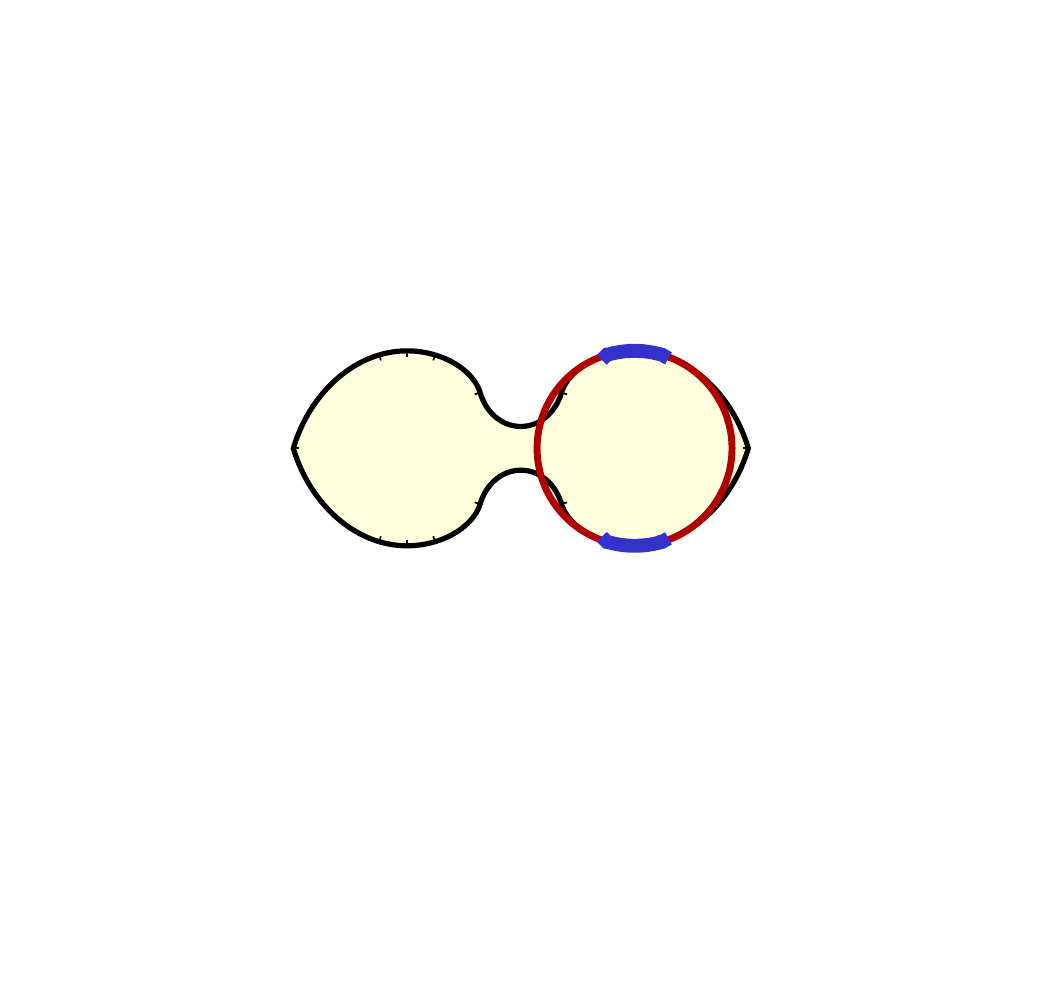}}
\hspace{-30pt}
&
\scalebox{0.65}{\includegraphics{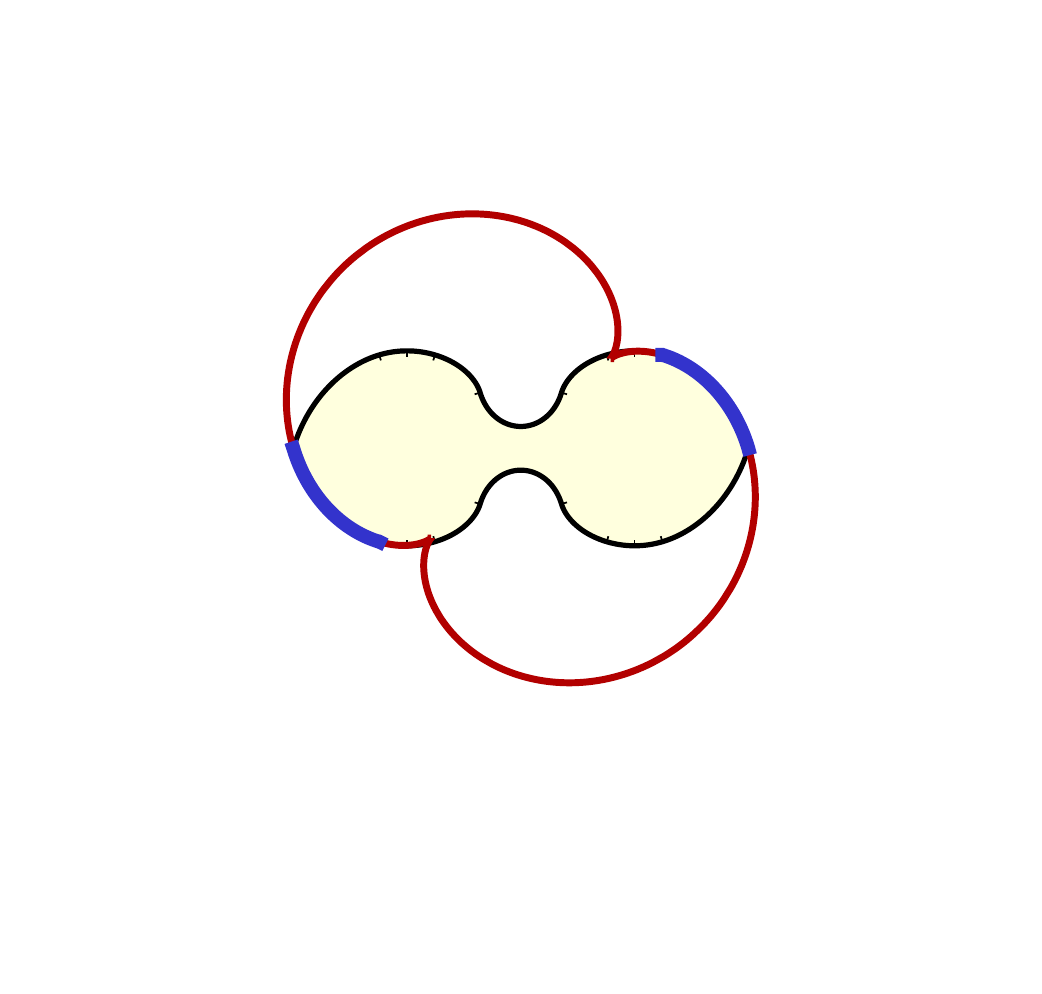}}
\\[-50pt] (a)&(b)
\\[20pt]
\hspace{-30pt}
\scalebox{0.65}{\includegraphics{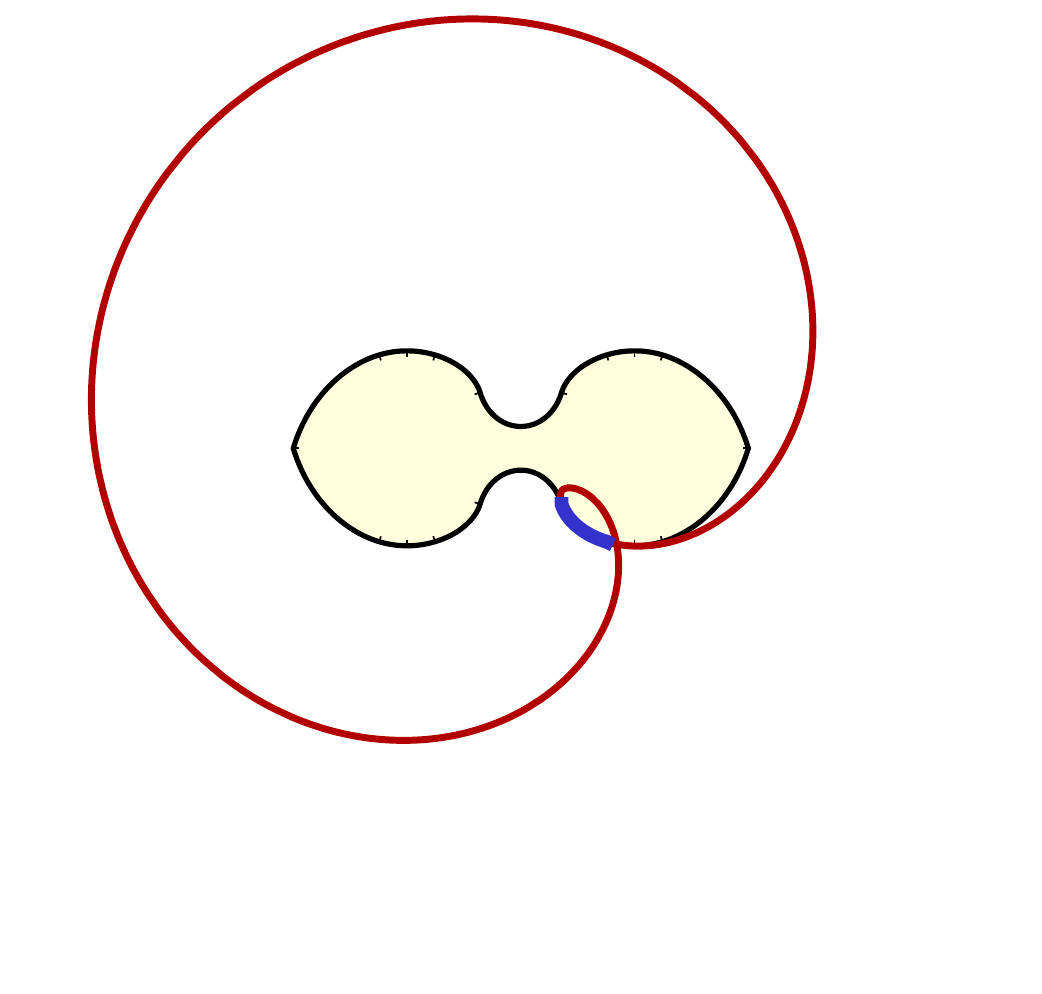}}
\hspace{-30pt}
&
\scalebox{0.65}{\includegraphics{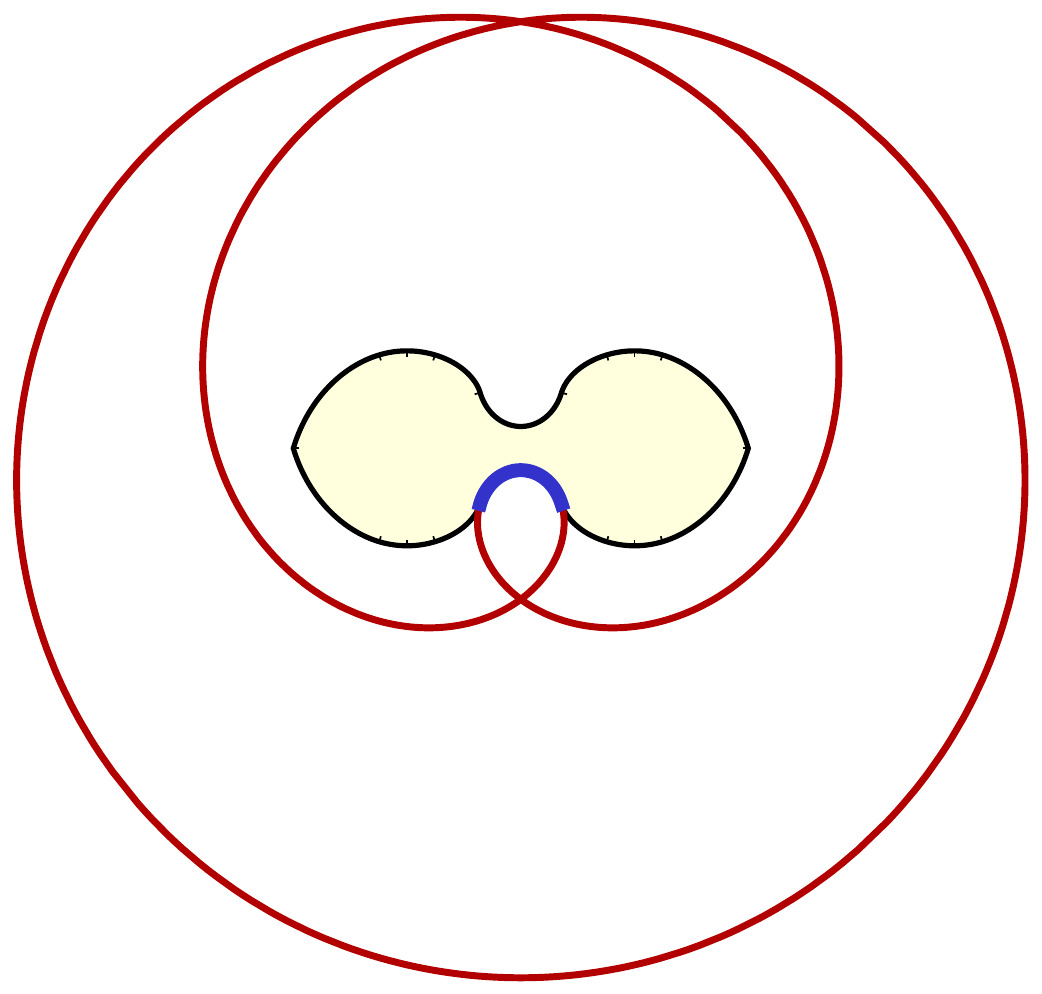}}
\\[5pt] (c)&(d)
\end{tabular}

\caption{The segments of the boundary of the shape $\Sigma$ are pieces of algebraic curves. Here we see the four distinct algebraic curves that appear as analytic continuations of boundary segments (along with their various symmetric reflections), consisting of: (a) a circle, and the three sextic algebraic curves: (b) $P(X,Y)=0$; (c) $Q(X,Y)=0$; (d) $R(X,Y)=0$.
}
\label{fig:algebraic-curves}
\end{center}
\end{figure}


The algebraic curves and their relations to the boundary segments of $\Sigma$ are shown in Fig.~\ref{fig:algebraic-curves}. Note also that since the different curve segments satisfy algebraic equations over an algebraic extension field of $\mathbb{Q}$, they also satisfy algebraic equations of higher degree with \emph{integer} coefficients.

\section{Open problems}

\label{sec:final-remarks}

We conclude with a few open problems:
\begin{enumerate}
\item Prove that Gerver's sofa and our shape $\Sigma$ are \emph{local} maxima of the area functional in the moving sofa problem and ambidextrous moving sofa problem, respectively.

\item Do there exist other locally area-maximizing shapes? For example, is there an asymmetric version of Gerver's sofa --- that is, a construction that follows the pattern \eqref{eq:gerver-contact-phases} and is obtained by gluing together the functions $\mathbf{x}_j(t)$, $1\le j\le 5$, except that the transitions between the five types of contact point sets occur at four successive angles $\varphi, \theta, \eta, \tau$, where 
$$0<\varphi<\theta<\pi/4<\eta<\tau<\pi/2,$$ and it is not the case that $\eta=\pi/2-\theta$ and $\tau=\pi/2-\varphi$? Is there a version of Gerver's sofa (symmetric or asymmetric) in which
for the third phase of rotation when $\theta<t<\pi/2-\theta$, the assumption that $\contact_{\mathbf{x}}(t)=\{ \mathbf{x},\mathbf{A},\mathbf{C} \}$ is replaced by the modified condition 
$\contact_{\mathbf{x}}(t)=\{ \mathbf{x},\mathbf{A}, \mathbf{B}, \mathbf{C}, \mathbf{D} \}$ (corresponding to Case~6 of Theorem~\ref{thm:odes} instead of Case~3 as in Gerver's construction)?

\item Can the assumption that the rotation path is ``well-behaved'' in Theorem~\ref{thm:odes} be weakened or removed?

\item Are there other natural variants of the moving sofa problem that give rise to shapes that can be expressed in closed form and/or are piecewise algebraic?

\end{enumerate}

\end{document}